\documentclass[12pt]{article}
\usepackage{amssymb}
\usepackage{amsmath}
\usepackage{amsthm}

\newtheorem{theorem}{Theorem}[section]
\newtheorem{lemma}[theorem]{Lemma}
\newtheorem{proposition}[theorem]{Proposition}

\newtheorem{corollary}[theorem]{Corollary}

\numberwithin{equation}{section}

\newtheorem{definition}[theorem]{Definition}

\newtheorem{remark}[theorem]{Remark}

\date{}

\begin{document}

\title{Lie groups, algebraic groups and lattices}

\author{Alexander Gorodnik}

\maketitle

\begin{abstract}
This is a brief introduction to the theories of Lie groups, algebraic groups and their
discrete subgroups, which is based on a lecture series given during the Summer School
held in the Banach Centre in Poland in Summer 2011.
\end{abstract}

{\small
\tableofcontents
}

\vspace{1cm}

This exposition is an expanded version of the 10-hour course 
given during the first week of the Summer School
``Modern dynamics and interactions with analysis, geometry and number theory''
that was held in the Bedlewo Banach Centre in Summer 2011.
The aim of this course was to cover background material regarding
Lie groups, algebraic groups and their discrete subgroups that would be useful in
the subsequent advanced courses. The presentation is intended to be accessible
for beginning PhD students, and we tried to make most emphasise
on ideas and techniques that play fundamental role in the theory of dynamical systems.
Of course, the notes would only provide one of the first steps towards mastering these topics,
and in \S8 we offer some suggestions for further reading.

In \S1 we develop the theory of (matrix) Lie groups. In particular, we introduce
the notion of Lie algebra, discuss relation between Lie-group homomorphisms
and the corresponding Lie-algebra homomorphisms, show that every Lie group has a structure
of an analytic manifold, and prove that every continuous homomorphism between Lie groups
is analytic. In \S2 we establish existence and uniqueness of invariant measures
on Lie groups. In \S3 we discuss finite-dimensional representations of Lie groups.
This includes a theorem regarding triangularisation of  representations of solvable groups
and a theorem regarding complete reducibility of representations of semisimple groups.
The later is treated using existence of the invariant measure constructed in \S2. 
Next, in \S4 we develop elements of the theory
of algebraic groups. We shall demonstrate that orbits for actions of algebraic
groups exhibit quite rigid behaviour, which is responsible for some of the rigidity phenomena
in the theory of dynamical systems. In \S5 we introduce the notion of a lattice in a Lie group
that plays important role in the theory of dynamical systems.
In particular, lattices can be used to construct homogeneous spaces of finite volume
leading to a rich class of dynamical systems, which are usually called the homogeneous dynamical systems.
Classification of smooth actions of higher-rank lattices is an active topic of research now.
In \S5 we present Poincare's geometric construction of lattices in $\hbox{SL}_2(\mathbb{R})$,
and in \S6 we explain number-theoretic constructions of lattices
which use the theory of algebraic groups.
These arithmetic lattices play crucial role in many applications of dynamical systems to number theory.
Finally, in \S7 we illustrate utility of techniques developed in these notes
by giving a dynamical proof of the Borel density theorem.

\section{Lie groups and Lie algebras}

\subsection{Lie groups and one-parameter groups}

Thought out these notes, $\hbox{M}_d(\mathbb{R})$ denotes the set of $d\times d$ matrices with 
real coefficients, and $\hbox{GL}_d(\mathbb{R})$ denotes the group of non-degenerate matrices.
The space $\hbox{M}_d(\mathbb{R})$ is equipped with the Euclidean topology,
and distance between the matrices will be measured by the norm:
$$
\|X\|=\sqrt{\sum_{i,j=1}^d |x_{ij}|^2},\quad X\in \hbox{M}_d(\mathbb{R}).
$$

\begin{definition} 
{\rm 
A (matrix) \emph{Lie group} is a closed subgroup of $\hbox{\rm GL}_d(\mathbb{R})$. 
}
\end{definition}

For instance, the following well-known matrix groups are examples of Lie groups:
\begin{itemize}
\item $\hbox{\rm SL}_d(\mathbb{R})=\{g\in \hbox{\rm GL}_d(\mathbb{R}):\, \det(g)=1\}$ --- the special
  linear group,
\item $\hbox{\rm O}_d(\mathbb{R})=\{g\in \hbox{\rm GL}_d(\mathbb{R}):\, {}^tgg=I\}$ --- the orthogonal group.
%\item $\hbox{\rm SL}_d(\mathbb{Z})$.
\end{itemize}

In order to understand the structure of Lie groups, we first study  one-parameter groups.

\begin{definition}
{\rm 
A \emph{one-parameter group} $\sigma$ is a continuous homomorphism $\sigma:\mathbb{R}\to \hbox{\rm GL}_d(\mathbb{R})$.
}
\end{definition}

One-parameter groups can be constructed using the exponential map:
$$
\exp(A)=\sum_{n=0}^\infty \frac{A^n}{n!}, \quad A\in \hbox{M}_d(\mathbb{R}).
$$
Since $\|A^n\|\le \|A\|^n$, this series converges uniformly on compact sets and defines an analytic map.

The exponential map satisfies the following properties:

\begin{lemma}\label{l:exp}
\begin{enumerate}
\item[(i)] $\exp({}^tA)={}^t\exp(A)$,
\item[(ii)] For $g\in \hbox{\rm GL}_d(\mathbb{R})$, $\exp(g Ag^{-1})=g\exp(A)g^{-1}$,
\item[(iii)] If $AB=BA$, then $\exp(A+B)=\exp(A)\exp(B)$,
\item[(iv)] $\det(\exp(A))=\exp(\hbox{\rm Tr}(A))$.
\end{enumerate}
\end{lemma}

\begin{proof}
(i) and (ii) are easy to check by a direct computation.

To prove (iii), we observe that by the Binomial Formula,
\begin{align*}
\exp(A)\exp(B)&=\sum_{n,m=0}^\infty \frac{A^nB^m}{n!m!}=\sum_{\ell=0}^\infty \frac{1}{\ell !}\left(\sum_{m+n=\ell} \frac{\ell!}{n!
    m!} A^n B^m \right)\\
&=\sum_{\ell=0}^\infty \frac{1}{\ell !} (A+B)^\ell =\exp(A+B). 
\end{align*}

To prove (iv), we use the Jordan Canonical Form. Every matrix $A$ can be written as 
$$
A=g(A_1+A_2)g^{-1},
$$
where $g\in\hbox{\rm GL}_d(\mathbb{R})$, $A_1$ is a diagonal matrix, and $A_2$ is an upper triangular
nilpotent matrix that commutes with $A_1$. It follows from (ii)--(iii) that once the claim is established for $A_1$ and
$A_2$, then it will also hold for  $A$. Since $A_1$ and $A_2$ are of special shape, the claim for them
can be verified by a direct computation.
\end{proof}

Lemma \ref{l:exp} implies that 
$$
\sigma_A(t)=\exp(tA)
$$
defines a one-parameter group.
We note that in a neighbourhood of zero,
$$
\sigma_A(t)=I+ t A+O(t^2).
$$
This implies that 
$$
\sigma_A'(0)=A\quad\hbox{and}\quad (D\exp)_0=I,
$$
where $D F$ denotes the derivative of a map $F:\hbox{M}_d(\mathbb{R})\to\hbox{M}_d(\mathbb{R})$.
Hence, by the Inverse Function Theorem,
the exponential map gives an 
analytic bijection from a small neighbourhood of the zero matrix $0$ to
a small neighbourhood of the identity matrix $I$ in $\hbox{M}_d(\mathbb{R})$.
This observation will play important role below.

\vspace{0.3cm}

Our first main result is a complete description of one-parameter groups:

\begin{theorem}\label{th:one-par}
Every one-parameter group is of the form $t\mapsto \exp(tA)$ for some $A\in \hbox{\rm M}_d(\mathbb{R})$.
\end{theorem}

This theorem, in particular, implies a non-obvious fact that every continuous homomorphism 
$\mathbb{R}\to \hbox{\rm GL}_d(\mathbb{R})$ is automatically analytic. As we shall see,
this is a prevalent phenomenon in the world of Lie groups (cf. Corollary \ref{c:annn} below).

\begin{proof}
We claim that if for some matrices $Y_1$ and $Y_2$, we have
$$
\|Y_1-I\|<1,\quad \|Y_2-I\|<1,\quad Y_1^2=Y_2^2,
$$
then $Y_1=Y_2$. Let us write $Y_i=I+A_i$.
Then since $(I+A_1)^2=(I+A_2)^2$,
$$
2A_1-2A_2=A_2^2-A_1^2=A_2(A_2-A_1)+(A_2-A_1)A_1,
$$
and
$$
2\|A_1-A_2\|\le (\|A_2\|+\|A_1\|)\|A_2-A_1\|.
$$
Because $\|A_2\|+\|A_1\|<2$, this implies the claim.

Let $\sigma$ be a one-parameter group. It follows from continuity of the maps $\sigma$ and $\exp$ that
there exist $\delta,\epsilon>0$ such that
$$
\sigma([-\epsilon,\epsilon])\subset \exp(\{\|X\|<\delta\})\subset \{\|Y-I\|<1\}.
$$
In particular, $\sigma(\epsilon)=\exp(\epsilon A)$ for some $A$ with $\|A\|<\delta/\epsilon$.
Then $\sigma(\frac{1}{2}\epsilon)^2=\exp(\frac{1}{2}\epsilon A)^2$, and applying the above claim,
we deduce that
$\sigma(\frac{1}{2}\epsilon)=\exp(\frac{1}{2}\epsilon A)$. We repeat this argument to conclude that
$\sigma(\frac{1}{2^m}\epsilon)=\exp(\frac{1}{2^m}\epsilon A)$ for all $m\in \mathbb{N}$, and taking powers,
we obtain $\sigma(\frac{n}{2^m}\epsilon)=\exp(\frac{n}{2^m}\epsilon A)$ for all $n\in\mathbb{Z}$ and
$m\in\mathbb{N}$. Therefore, it follows from continuity of the maps $\sigma$ and $\exp$ that 
$\sigma(t\epsilon)=\exp(t\epsilon A)$ for all $t\in \mathbb{R}$, as required.
\end{proof}

\subsection{Lie algebras}

One of the most basic and very useful ideas in mathematics is the idea of linearisation.
In the setting of Lie groups, this leads to the notion of Lie algebra. 
For $X,Y\in \hbox{M}_d(\mathbb{R})$, we define the {\it Lie bracket} by
$$
[X,Y]=XY-YX.
$$
It turns out that the Lie bracket corresponds to the second order term of 
the Taylor expansion of product map $(g,h)\mapsto g\cdot h$.

\begin{definition}
{\rm 
A subspace of $\hbox{M}_d(\mathbb{R})$ is called a \emph{Lie algebra} if it closed
with respect to the Lie bracket operation.
}
\end{definition}

\begin{definition}
{\rm 
The \emph{Lie algebra of a Lie group $G$} is defined by 
$$
\mathcal{L}(G)=\{X\in \hbox{M}_d(\mathbb{R}):\, \exp(tX)\in G\quad\hbox{for all $t\in \mathbb{R}$} \}.
$$
}
\end{definition}

For example, using Lemma \ref{l:exp}, one can check that
\begin{itemize}
\item $\mathcal{L}(\hbox{\rm SL}_d(\mathbb{R}))=\{X\in \hbox{\rm M}_d(\mathbb{R}):\, \hbox{Tr}(X)=0\}$,
\item $\mathcal{L}(\hbox{\rm O}_d(\mathbb{R}))=\{X\in \hbox{\rm M}_d(\mathbb{R}):\, {}^tX+X=0\}$.
\end{itemize}

We prove that

\begin{proposition}\label{p:lie_algebra}
$\mathcal{L}(G)$ is a Lie algebra, namely, it is a vector space and is closed under the Lie bracket operation.
\end{proposition}

Given $A,B\in\hbox{M}_d(\mathbb{R})$ such that $\|A\|,\|B\|<r$ with $r\approx 0$,
the product $\exp(A)\exp(B)$ is contained in a small neighbourhood of identity.
Hence, 
$$
\exp(A)\exp(B)=\exp(C),
$$
where $C=C(A,B)$ is a uniquely determined matrix contained
in a neighbourhood of zero. We compute the Taylor expansion for $C$:

\begin{lemma}\label{l:c}
$C(A,B)=A+B+\frac{1}{2}[A,B]+O(r^3)$.
\end{lemma}

\begin{proof}
We have $\exp(A)=I+O(r)$ and $\exp(B)=I+O(r)$, so that
$$
\exp(A)\exp(B)=I+O(r)\quad\hbox{and}\quad C=\exp^{-1}(I+O(r))=O(r).
$$
This implies that $\exp(C)=I+C+O(r^2)$. On the other hand,
$$
\exp(A)\exp(B)=(I+A+O(r^2))(I+B+O(r^2))=I+A+B+O(r^2).
$$
Therefore, $C=A+B+O(r^2)$. 

This process can be continued to compute the higher order terms in the expansion of $C$.
We write $C=A+B+S$ where $S=O(r^2)$. Then
\begin{align*}
\exp(C)&=I+(A+B+S)+(A+B+S)^2/2+O(r^3)\\
&=I+A+B+S+(A+B)^2/2+O(r^3).
\end{align*}
On the other hand,
\begin{align*}
\exp(A)\exp(B)&=(I+A+A^2/2+O(r^3))(I+B+B^2/2+O(r^3))\\
&=I+A+B+AB+A^2/2+B^2/2+O(r^3).
\end{align*}
Hence,
\begin{align*}
S&=(I+A+B+AB+A^2/2+B^2/2)-(I+A+B+(A+B)^2/2)+O(r^3)\\
&=\frac{1}{2}[A,B]+O(r^3).
\end{align*}
This implies the lemma. 
\end{proof}

The proof of Lemma \ref{l:c} can be generalised to prove the
\emph{Campbell--Baker--Hausdorff formula}:
\begin{equation}\label{eq:baker}
C(A,B)=\sum_{n=0}^N C_n(A,B)+O(r^{N+1}),
\end{equation}
where $C_n$ are explicit homogeneous polynomials of degree $n$
which are expressed in terms of Lie brackets.

For Lemma \ref{l:c}, we deduce:

\begin{corollary}\label{c:prod}
For every $A,B\in \hbox{\rm M}_d(\mathbb{R})$,
\begin{enumerate}
\item[(i)] $\exp(A+B)=\lim_{n\to\infty} (\exp(A/n)\exp(B/n))^n$,
\item[(ii)] $\exp([A,B])=\lim_{n\to\infty} (\exp(A/n)\exp(B/n)\exp(-A/n)\exp(-B/n))^{n^2}$.
\end{enumerate}
\end{corollary}

\begin{proof}
By Lemma \ref{l:c},
\begin{equation}\label{eq:cc_1}
\exp(A/n)\exp(B/n)=\exp(C_n),\quad\hbox{where $C_n=(A+B)/n+O(1/n^2)$.}
\end{equation}
Hence,
$$
(\exp(A/n)\exp(B/n))^n=\exp(A+B+O(1/n))\to \exp(A+B)
$$
as $n\to\infty$. This proves (i). 

The proof of (ii) is similar and is left to the reader. 
\end{proof}

Now we are ready to prove Proposition \ref{p:lie_algebra}.

\begin{proof}[Proof of Proposition \ref{p:lie_algebra}]
It is clear for the definition that if $A\in \mathcal{L}(G)$, then $\mathbb{R}A\subset \mathcal{L}(G)$.
Hence, it remains to show that for $A,B\in \mathcal{L}(G)$, 
the matrices $A+B$ and $[A,B]$ also belong to $\mathcal{L}(G)$.
%Without loss of generality, we may assume that $A$ and $B$ are sufficiently small,
%so that the exponential map is one-to-one in a neighbourhhod.

We shall use the following observation:
\begin{equation}\label{eq:obs}
\hbox{if $C_n\in \exp^{-1}(G)$, $C_n\to 0$, $s_n C_n\to D$ for some $s_n\in\mathbb{R}$,
then $D\in \mathcal{L}(G)$.}
\end{equation}
To prove this observation, we need to show that $\exp(tD) \in G$
for all $t\in\mathbb{R}$. Let $m_n=\lfloor ts_n \rfloor$. Then 
$$
\lim_{n\to\infty} m_n C_n=  \lim_{n\to\infty} ts_n C_n=tD.
$$
Since $C_n\in \exp^{-1}(G)$, we have $\exp(m_n C_n)=\exp(C_n)^{m_n}\in G$. Hence, since $G$ is closed,
$\exp(tD)\in G$, which proves the observation.

Now let us prove that $A+B\in \mathcal{L}(G)$ when $A,B\in \mathcal{L}(G)$. For sufficiently large $n$,
$$
\exp(A/n)\exp(B/n)=\exp(C_n),
$$
where $C_n\to 0$. It is clear that $C_n\in \exp^{-1}(G)$.
By (\ref{eq:cc_1}), $nC_n\to A+B$. Hence, the above observation implies that
$A+B\in\mathcal{L}(G)$.

Similarly,  using Corollary \ref{c:prod}(ii), we obtain that if $A,B\in \mathcal{L}(G)$, then
$$
\exp(A/n)\exp(B/n)\exp(-A/n)\exp(-B/n)=\exp(C_n),
$$
where $C_n\in \exp^{-1}(G)$, $C_n\to 0$ and $n^2 C_n\to [A,B]$.
Therefore, the above observation shows that $[A,B]\in \mathcal{L}(G)$.
\end{proof}

The exponential map can be used to show that a Lie group locally looks like the Euclidean space
of dimension $\dim(\mathcal{L}(G))$ and has a structure of an analytic manifold.

\begin{proposition}\label{p:nbhd}
The exponential map defines a bijection between a neighbourhood of zero in $\mathcal{L}(G)$
and a neighbourhood of identity in $G$.
\end{proposition}

\begin{proof}
We write $\hbox{M}_d(\mathbb{R})=\mathcal{L}(G)\oplus V$, where $V$ is a complementary subspace,
and denote by $\pi: \hbox{M}_d(\mathbb{R})\to\mathcal{L}(G)$ and
$\bar \pi: \hbox{M}_d(\mathbb{R})\to V$ the corresponding projection maps.
Let
$$
F:\hbox{M}_d(\mathbb{R})\to \hbox{M}_d(\mathbb{R}): A\mapsto \exp(\pi(A))\exp(\bar\pi(A)).
$$
We have
$$
\frac{d}{dt}(\exp(\pi(tA))\exp(\bar\pi(tA)))|_{t=0}=\pi(A)+\bar\pi(A)=A.
$$
Hence, $(DF)_0=I$, and it follows from the Inverse Function Theorem that
for sufficiently small neighbourhood $\mathcal{O}$ of $0$, the map $F:\mathcal{O}\to F(\mathcal{O})$
is a bijection. 

We already remarked above that the exponential map defines a bijection between a neighbourhood of 
$0$ in $\hbox{M}_d(\mathbb{R})$ and a neighborhood of $I$ in $\hbox{M}_d(\mathbb{R})$. 
To prove the proposition, it remains to show that
$\exp(\mathcal{O}\cap \mathcal{L}(G))\subset G$ is a neighbourhood of identity in $G$.
Suppose on the contrary that the set $\exp(\mathcal{O}\cap \mathcal{L}(G))$ 
is not a neighbourhood of identity in $G$. Then since the set $\exp(\mathcal{O})\cap G$
is a neighbourhood of identity in $G$, it follows that there exists a sequence $B_n\to 0$
such that $\exp(B_n)\in G$ and $B_n\notin \mathcal{L}(G)$. We can write $\exp(B_n)=F(A_n)$
with some matrix $A_n$ such that $A_n\to 0$. Since $B_n\notin \mathcal{L}(G)$, we have $\bar \pi(A_n)\ne 0$.
We note that 
$$
\exp(\bar \pi(A_n))=\exp(\pi(A_n))^{-1}\exp(B_n)\in G.
$$
After passing to a subsequence, we may assume that $\frac{\bar \pi(A_n)}{\|\bar \pi(A_n)\|}\to C$
for some matrix $C\in \hbox{M}_d(\mathbb{R})$ with $\|C\|=1$. It is clear that $C\in V$. 
On the other hand, it follows from the observation (\ref{eq:obs})
that $C\in \mathcal{L}(G)$. This contradiction completes the proof of the proposition.
\end{proof}

\begin{remark}\label{r:lie}
{\rm 
The proof of Proposition \ref{p:nbhd} shows that in a neighbourhood of identity, $G$ coincides  
with the zero locus $\bar \pi \circ F^{-1}$. Moreover, $\mathcal{L}(G)$  is the tangent space of 
this locus at identity.
}
\end{remark}

Proposition \ref{p:nbhd} can be used to define a \emph{manifold structure} on a Lie group $G$.
We fix a neighbourhood $\mathcal{U}$ of zero in $\hbox{M}_d(\mathbb{R})$ such that $\exp$ is an analytic 
bijection $\mathcal{U}\to \exp(\mathcal{U})$ and set $\mathcal{O}=\mathcal{L}(G)\cap \mathcal{U}$.
For every $g\in G$, we define a \emph{coordinate chart} around $g$ by
$$
\phi:\mathcal{O}\to G: x\mapsto g\exp(x).
$$ 
This coordinate chart defines a bijection between $\mathcal{O}$ and a neighbourhood of $g$.
If $\psi:\mathcal{O}\to G$ is another coordinate chart, then the map
\begin{equation}\label{eq:change}
\psi^{-1}\phi: \phi^{-1}(\phi(\mathcal{O})\cap \psi(\mathcal{O}))\to \psi^{-1}(\phi(\mathcal{O})\cap
\psi(\mathcal{O}))
\end{equation}
is analytic. We say that a map $f:G\to \mathbb{R}^k$ is \emph{analytic} if $f\circ \phi$ is analytic
for all coordinate charts. In particular, the product map $G\times G\to G:(g_1,g_2)\mapsto g_1g_2$
and the inverse map $G\to G: g\mapsto g^{-1}$ are analytic.
Now a Lie group $G$ can be considered as collection of coordinate charts
which are glued together according to the maps (\ref{eq:change}) and such that
the group operations are analytic. This leads to the notion of an abstract Lie group.
For simplicity of exposition, we restrict our discussion to matrix Lie groups.

\subsection{Lie-group homomorphisms}

In this section we study continuous homomorphisms $f:G_1\to G_2$ between Lie groups.
We show that they induce a Lie-algebra homomorphisms between the corresponding Lie algebras,
and that every continuous homomorphism is automatically analytic.

\begin{theorem}\label{th:hom}
Let $f:G_1\to G_2$ be a continuous homomorphism between Lie groups $G_1$ and $G_2$.
Then there exists a Lie-algebra homomorphism $Df:\mathcal{L}(G_1)\to \mathcal{L}(G_2)$
such that 
$$
\exp(Df(X))=f(\exp(X))\quad \hbox{for all $X\in \mathcal{L}(G_1)$.}
$$
\end{theorem}

\begin{proof}
For every $X\in \mathcal{L}(G_1)$, the map $t\mapsto f(\exp(tX))$ is a one-parameter subgroup.
Hence, by Theorem \ref{th:one-par}, we have $f(\exp(tX))=\exp(tY)$ for some $Y\in\hbox{M}_d(\mathbb{R})$.
Since $f(G_1)\subset G_2$, we obtain $Y\in\mathcal{L}(G_2)$.
It is also clear that such $Y$ is uniquely defined.
We set $Df(X)=Y$. It follows from the definition that
\begin{equation}\label{eq:scale}
Df(sX)= s Df(X)\quad\hbox{for all $s\in \mathbb{R}$.}
\end{equation}
We claim that $Df$ is a Lie-algebra homomorphism, namely,
we need to check that for every $X_1,X_2\in\mathcal{L}(G_1)$,
\begin{align}\label{eq:id}
Df(X_1+X_2)&=Df(X_1)+Df(X_2),\\
Df([X_1,X_2])&=[Df(X_1),Df(X_2)].
\end{align}

To verify the first identity, we use Corollary \ref{c:prod}(i) and continuity of $f$:
\begin{align*}
\exp(Df(X_1+X_2))&=f(\exp(X_1+X_2))=\lim_{n\to\infty} f((\exp(X_1/n)\exp(X_2/n))^n)\\
&=\lim_{n\to\infty} (f(\exp(X_1/n))f(\exp(X_2/n)))^n\\
&=\lim_{n\to\infty} (\exp(Df(X_1)/n))\exp(Df(X_2)/n))^n\\
&=\exp(Df(X_1)+Df(X_2)).
\end{align*}
Because of (\ref{eq:scale}), it is sufficient to verify (\ref{eq:id}) when
$X_1$ and $X_2$ are sufficiently small. Then the exponential map is one-to-one,
and the first identity follows. 

The second identity can be proved similarly
with a help of Corollary \ref{c:prod}(ii).
\end{proof}

Since the exponential map is analytic, Theorem~\ref{th:hom}
implies

\begin{corollary}\label{c:annn}
Any continuous homomorphism between Lie groups is analytic.
\end{corollary}

In view of Theorem \ref{th:hom}, it is natural to ask whether
every Lie-algebra homomorphism $F:\mathcal{L}(G_1)\to \mathcal{L}(G_2)$
corresponds to a homomorphism $f:G_1\to G_2$ of the corresponding Lie groups.
As the following example demonstrates, this is not always the case.
Let
$$
G=\hbox{O}_2(\mathbb{R})=\{g:\, {}^t g g=I\}
=\left\{
\pm \left(
\begin{tabular}{rr}
$\cos\theta$ & $\sin \theta$ \\
$-\sin \theta$ & $\cos\theta$
\end{tabular}
\right): \theta\in [0,2\pi)
\right\}.
$$ 
Its Lie algebra
$$
\mathcal{L}(G)=\{x:\, {}^t X+ X=0\}
=\left\{
\left(
\begin{tabular}{ll}
$0$ & $\theta$ \\
$-\theta$ & $0$
\end{tabular}
\right): \theta\in \mathbb{R}
\right\}
$$ 
has trivial Lie bracket operation, and every linear map $\theta\mapsto c\, \theta$
defines a Lie-algebra homomorphism $\mathcal{L}(G)\to \mathcal{L}(G)$. 
However, this linear map corresponds to a homomorphism
$G\to G$ only when $c\in \mathbb{Z}$. This example demonstrates that
the Lie algebra captures only local structure of its Lie group.

It turns out that for simply connected Lie groups the answer to the above question is positive.
Recall that 

\begin{definition}\label{def:simply_connected}
{\rm 
A topological space $X$ is called \emph{simply connected} if $X$ is path connected
and for any two paths between $x_0,x_1\in X$ can be continuously deformed into each other,
namely, for any continuous maps $\alpha_0,\alpha_1:[0,1]\to X$ such that $\alpha_0(0)=\alpha_1(0)=x_0$
and $\alpha_0(1)=\alpha_1(1)=x_1$, there exists a continuous map $\alpha:[0,1]^2\to X$ such that
$$
\alpha(0,\cdot)=\alpha_0,\;\; \alpha(1,\cdot)=\alpha_1,\;\; \alpha(\cdot, 0)=x_0,\;\; \alpha(\cdot, 1)=x_1.
$$
}
\end{definition}

\begin{theorem}\label{th:simply}
If a Lie group $G_1$ is simply connected and $F:\mathcal{L}(G_1)\to \mathcal{L}(G_2)$
is a Lie-algebra homomorphism, then there exists a smooth homomorphism $f:G_1\to G_2$ 
such that $F=Df$.
\end{theorem}

\begin{proof}
We fix a small neighbourhood $\mathcal{U}$ of identity in $G_1$
such that the exponential map defines a bijection on $\mathcal{U}$.
We define 
\begin{equation}\label{eq:f}
f(g)=\exp(F(\exp^{-1}(g)))\quad\hbox{for $g\in \mathcal{U}$}.
\end{equation}
In order to define $f$ for general $g\in G_1$ we take a continuous
path $\alpha:[0,1]\to G_1$ from $I$ to $g$ and take a partition $\{t_i\}_{i=0}^m$
of $[0,1]$ such that $\alpha(t_{i+1})\in \mathcal{U}\alpha(t_i)$. Then
$$
g=\alpha(1)=\alpha(t_m)\alpha(t_{m-1})^{-1}\cdots \alpha(t_1)\alpha(t_0)^{-1}.
$$
We define 
\begin{equation}\label{eq:fg}
f(g)=f(\alpha(t_m)\alpha(t_{m-1})^{-1})\cdots f(\alpha(t_1)\alpha(t_0)^{-1}).
\end{equation}
We shall show that this definition does not depend on the choices of the path and the partition.
Take a neighbourhood $\mathcal{V}$ of identity in $G_1$, and 
let us consider a continuous path $\beta:[0,1]\to G_1$ which is a continuous perturbation of $\alpha$
defined as follows. We replace the map $\alpha$ on one of the intervals
$[t_i,t_{i+1}]$ by another map such that
for some $s\in (t_i,t_{i+1})$, we have 
$$
\beta(t_{i+1})\beta(s)^{-1}, \beta(s)\beta(t_{i})^{-1}\in \mathcal{V},
$$
and refine the partition by adding the point $s$.
This gives the same value $f(g)$ if 
\begin{equation}\label{eq:beta}
f(\beta(t_{i+1})\beta(s)^{-1})f(\beta(s)\beta(t_{i})^{-1})
=f(\beta(t_{i+1})\beta(t_{i})^{-1}).
\end{equation}
We write 
$$
\beta(t_{i+1})\beta(s)^{-1}=\exp(X)\quad\hbox{and}\quad \beta(s)\beta(t_{i})^{-1}=\exp(Y)
$$
for some $X,Y\in \exp^{-1}(\mathcal{V})$.
Then $\beta(t_{i+1})\beta(t_{i})^{-1}=\exp(X)\exp(Y)$.
We apply the Baker--Campbell-Hausdorff formula (\ref{eq:baker}).
Assuming that $\mathcal{V}$ is sufficiently small, we obtain
\begin{align*}
f(\exp(X)\exp(Y))&= f(\exp(C(X,Y)))=\exp(F(C(X,Y)))\\
&=\exp(C(F(X),F(Y)))=\exp(F(X))\exp(F(Y))\\
&=f(\exp(X))f(\exp(Y)).
\end{align*}
This proves (\ref{eq:beta}). In particular, it is clear from the argument
that the definition of $f(g)$ in (\ref{eq:fg}) is independent of the partition.

Since $G_1$ is simply connected, given two paths $\alpha_0$ and $\alpha_1$ from $I$ to $g$, 
we can transform $\alpha_0$ to $\alpha_1$ using finitely many perturbations as above.
Hence, the definition of $f(g)$ in (\ref{eq:fg}) is independent of the path,
and we have a well-defined map $f:G_1\to G_2$.

Now we show that $f$ is a homomorphism. Let $g,h\in G$ and $\alpha,\beta:[0,1]\to G$
be paths from $I$ to $g$ and $h$ respectively. We define a path from $I$ to $gh$ by
$$
\gamma(t)=\left\{
\begin{tabular}{ll}
$\beta(2t)$, & $t\in [0,1/2]$,\\
$\alpha(2t-1)h,$ & $t\in [1/2,1]$.
\end{tabular}
\right.
$$
Then according to the definition of $f$, 
\begin{align*}
f(gh)&= \prod_{i=m}^1 f(\gamma(t_i)\gamma(t_{i-1})^{-1})\\
&=\left(\prod_{i=m}^{m'} f(\alpha(t_i)\alpha(t_{i-1})^{-1})\right) \left(\prod_{i=m'-1}^1
  f(\beta(t_i)\beta(t_{i-1})^{-1})\right)\\
&=f(g)f(h).
\end{align*}
Hence, $f$ is a homomorphism. 

Finally, the relation $Df=F$ follows from (\ref{eq:f}).
\end{proof}

\section{Invariant measures} 

The Lebesgue measure on $\mathbb{R}^d$ plays fundamental role in classical analysis.
It can be characterised uniquely (up to a scalar multiple) by the following properties:
\begin{itemize}
\item \emph{(invariance)} For every $f\in C_c(\mathbb{R}^d)$ and $a\in\mathbb{R}^d$,
$$
\int_{\mathbb{R}^d} f(a+x)\, dx= \int_{\mathbb{R}^d} f(x)\, dx.
$$
\item \emph{(local finiteness)} For every bounded measurable $B\subset \mathbb{R}^d$,
$$
\hbox{vol}(B)<\infty.
$$
\end{itemize}
In this section we discuss invariant measures for Lie groups. We first show that the invariant measure
is unique as in the case of the Lebesgue measure.

\begin{theorem}\label{th:unique}
A left-invariant locally finite Borel measure on a Lie group is unique up to a scalar multiple.
\end{theorem}

\begin{proof}
Let $m_1$ and $m_2$ be nonzero left-invariant locally finite Borel measures on a Lie group $G$.
Replacing $m_1$ by $m_1+m_2$, we may assume that $m_2$ is absolutely continuous with
respect to $m_1$. Namely, if $m_1(B)=0$ for some measurable $B\subset G$, then $m_2(B)=0$ as well. 

We fix a nonnegative $\phi\in C_c(G)$ with $\int_G \phi \, dm_1=1$ and set $c=\int_G\phi\, dm_2$.
Using the Fubini Theorem and invariance of the measures, we deduce that for every $f\in C_c(G)$,
\begin{align}\label{eq:Delta0}
c\cdot \int_G f(x)\, dm_1(x)&=\int_{G\times G} f(x)\phi(y)\, dm_1(x)dm_2(y)\\
&=\int_G\left(\int_Gf(x)\, dm_1(x)\right)\phi(y)\, dm_2(y)\nonumber \\
&=\int_G\left(\int_Gf(y^{-1}x)\, dm_1(x)\right)\phi(y)\, dm_2(y)\nonumber \\
&=\int_G\left(\int_Gf((x^{-1}y)^{-1})\phi(y)\, dm_2(y)\right)\, dm_1(x)\nonumber \\
&=\int_{G\times G} f(y^{-1})\phi(xy)\, dm_1(x)dm_2(y)\nonumber \\
&=\int_G f(y^{-1})\Delta(y)\, dm_2(y),\nonumber
\end{align}
where $\Delta(y)=\int_G \phi(xy)\, dm_1(x)$.
Applying the same argument with $m_2$ replaced by $m_1$ twice, we obtain
\begin{align}\label{eq:Delta}
1\cdot \int_G f(x)\, dm_1(x)&= \int_G f(y^{-1})\Delta(y)\, dm_1(y)\\
&= \int_G f(y)\Delta(y^{-1})\Delta(y)\, dm_1(y).\nonumber
\end{align}
Let $B=\{y:\, \Delta(y^{-1})\Delta(y)\ne 1\}$. 
Since (\ref{eq:Delta}) holds for all $f\in C_c(G)$, it follows that $m_1(B)=0$.
Then $m_2(B)=0$ as well, and applying (\ref{eq:Delta0})--(\ref{eq:Delta}), we get
\begin{align*}
\int_G f\, dm_2&=\int_G f(y)\Delta(y^{-1})\Delta(y)\, dm_2(y)\\
&= c\cdot \int_G f(y^{-1})\Delta(y)\, dm_1(y)\\
&= c\cdot \int_G f\, dm_1
\end{align*}
for every $f\in C_c(G)$. Because the measures $m_1$ and $m_2$ are locally finite,
one can show that they are uniquely determined by their values on $C_c(G)$.
Hence, $m_2= c\cdot m_1$.
\end{proof}

Our next task is to develop the theory of integration
using a collection of coordinate charts constructed in the previous section. Here 
we take the most elementary approach, but if the reader is familiar with
the theory of differential forms, most of this discussion might be redundant.

Let $\phi_1:\mathcal{O}_1\to G$ be a coordinate chart for $G$, where $\mathcal{O}_1$ is an open subset of
$\mathbb{R}^d$. We fix a measurable bounded function $\delta_1:\mathcal{O}_1\to\mathbb{R}^+$.
Given a function $f$ on $G$ with support contained in $\phi_1(\mathcal{O}_1)$. We define
$$
\int_G f\,  dm_{\delta_1}=\int_{\mathcal{O}_1} f(\phi_1(x))\delta_1(x)\, dx.
$$
This definition depends on the choices of the coordinate chart $\phi_1$ and the function $\delta_1$.
Let $\phi_2:\mathcal{O}_2\to G$ be another coordinate chart and $\delta_2:\mathcal{O}_2\to\mathbb{R}^+$.
Suppose that the support of $f$ is also contained in $\phi_2(\mathcal{O}_2)$.
Then using the change of variables formula for the Lebesgue integral, we obtain
\begin{align*}
\int_G f \, dm_{\delta_2}&=\int_{\mathcal{O}_2} f(\phi_2(y))\delta_2(y)\, dy\\
&=\int_{\mathcal{O}_1} f(\phi_1(x))\delta_2(\phi_2^{-1}\phi_1(x))\hbox{Jac}(\phi_2^{-1}\phi_1)_x\, dx,
\end{align*}
where $\hbox{Jac}(\phi_2^{-1}\phi_1)_x$ denotes the Jacobian of the map $\phi_2^{-1}\phi_1$.
Hence,
\begin{equation}\label{eq:comp}
m_{\delta_1}=m_{\delta_2}\quad \Longleftrightarrow\quad \delta_2(\phi_2^{-1}\phi_1(x))\hbox{Jac}(\phi_2^{-1}\phi_1)_x=\delta_1(x).
\end{equation}

\begin{definition}
{\rm
A \emph{volume density} $\delta$ is a collection of bounded measurable functions
$\delta_\phi:\mathcal{O}\to\mathbb{R}^+$ assigned to each coordinate chart $\phi:\mathcal{O}\to\mathbb{R}^+$
that satisfy the compatibility condition (\ref{eq:comp}).
}
\end{definition}

Now given a volume density $\delta$ on a Lie group $G$, we define a measure $m_\delta$ on $G$.
For every $f\in C_c(G)$, we write $f=f_1+\cdots +f_\ell$ for $f_i\in C_c(G)$ such that 
the support of $f_i$ is contained in $\phi_i(\mathcal{O}_i)$ for some coordinate charts
$\phi_i:\mathcal{O}_i\to G$. We define
$$
\int_G f\, dm_\delta= \sum_{i=1}^\ell \int_{\mathcal{O}_i} f(\phi_i(x))\delta_{\phi_i}(x)\, dx.
$$
One can check using the compatibility condition (\ref{eq:comp}) that this definition
is independent of the choices of the decomposition of $f$ and the coordinate charts $\phi_i$,
so that the measure $m_\delta$ is well-defined.

We investigate when the measure $m_\delta$ is left-invariant.
Given a function $f\in C_c(G)$ such that the support of $f$ is contained in $\phi(\mathcal{O})$
for a coordinate chart $\phi:\mathcal{O}\to G$, we have
$$
\int_G f\, dm_\delta=\int_{\mathcal{O}} f(\phi(x))\delta_\phi(x)\, dx
$$
To compute the integral of the function $x\mapsto f(g_0 x)$ with $g_0\in G$,
we observe that its support is contained
in $g_0^{-1}\phi(\mathcal{O})$, so that
\begin{align*}
\int_G f(g_0x)\, dm_\delta(x) &= \int_G f(g_0\, g_0^{-1}\phi(x))\delta_{g_0^{-1}\phi}(x)\, dx\\
&=\int_{\mathcal{O}} f(\phi(x)) \delta_{g_0^{-1}\phi}(x)\, dx.
\end{align*}
This computation shows that the measure $m_\delta$ is left-invariant if and only if
\begin{equation}\label{eq:inv}
\delta_\phi=\delta_{g_0\phi}\quad
\hbox{for all $g_0\in G$ and all coordinate charts $\phi$.}
\end{equation}

Using this construction, we prove

\begin{theorem}
Every Lie group supports an analytic left-invariant measure.
\end{theorem}

\begin{proof}
In view of the above discussion, it is sufficient to show that there exists
an analytic volume density satisfying (\ref{eq:comp}) and (\ref{eq:inv}). 
We fix a coordinate chart $\phi_0:\mathcal{O}_0\to G$ such that $\phi_0(z_0)=I$ for some $z_0\in
\mathcal{O}_0$. For every other coordinate chart $\phi:\mathcal{O}\to G$, we define
\begin{align*}
F_\phi(x,z)&=\phi^{-1}(\phi(x)\phi_0(z)),\\
\delta_\phi(x)&=\hbox{Jac}\left(F_\phi(x,\cdot)\right)^{-1}_{z_0}.
\end{align*}
Given any other coordinate chart $\psi$, we have 
$$
F_\psi(\psi^{-1}\phi(x),z)=\psi^{-1}(\phi(x)\phi_0(z))= \psi^{-1}\phi(F_\phi(x,z)),
$$
and
$$
\hbox{Jac}\left(F_\psi(\psi^{-1}\phi(x),\cdot) \right)_{z_0}=\hbox{Jac}(\psi^{-1}\phi)_x
\hbox{Jac}(F_\phi(x,\cdot)).
$$
This implies that (\ref{eq:comp}) holds.

To check (\ref{eq:inv}), we compute
\begin{align*}
F_{g_0\phi}(x,z)=(g_0\phi)^{-1}(g_0\phi(x)\phi_0(z))=\phi^{-1}(\phi(x)\phi_0(z))
=F_\phi(x,z),
\end{align*}
so that
$$
\hbox{Jac}(F_{g_0\phi}(x,\cdot))_{z_0}=\hbox{Jac}(F_{\phi}(x,\cdot))_{z_0},
$$
and (\ref{eq:inv}) holds. This completes the proof of the theorem.
\end{proof}

The above construction of the invariant measure is quite explicit.
For example, for the group
$$
\hbox{SL}_2(\mathbb{R})=\left\{\left(
\begin{tabular}{ll}
$a$ & $b$ \\
$c$ & $d$
\end{tabular}
\right): ad-bc=1\right\}
$$
the left-invariant measure is given by $\frac{dadbdc}{a}$. This measure is also right-invariant.
In general, a left-invariant measure does not have to be right-invariant.

\begin{definition}
{\rm 
A Lie group is called \emph{unimodular} if the left-invariant measure on $G$ is also right-invariant.
}
\end{definition}

For future reference we also prove 

\begin{proposition}\label{p:dec}
Let $G$ be a unimodular Lie group and $G=ST$ where $S$ and $T$ are closed subgroups such that $S\cap
T=1$. Then the invariant measure on $G$ is given by
$$
\int_G f\, dm=\int_{S\times T} f(st^{-1})\, dm_S(s) dm_T(t),\quad f\in C_c(G),
$$
where $m_S$ and $m_T$ are the left-invariant measures on $S$ and $T$ respectively.
\end{proposition}

\begin{proof}
The map $\Phi(s,t)=st^{-1}$, $(s,t)\in S\times T$, defines a homeomorphism between $S\times T$ and $G$.
We consider the measure on $S\times T$ defined by
$$
f\mapsto \int_G f(\Phi^{-1}(g))\, dm(g),
$$
where $m$ is the invariant measure on $G$. For $(s_0,t_0)\in S\times T$,
\begin{align*}
\int_G f((s_0,t_0)\cdot \Phi^{-1}(g))\, dm(g)&=\int_G f(\Phi^{-1}(s_0gt_0^{-1}))\, dm(g)\\
&= \int_G f(\Phi^{-1}(g))\, dm(g).
\end{align*}
Hence, it follows from the uniqueness of invariant measure (Theorem \ref{th:unique}) that
this measure is proportional to the product measure $m_S\times m_T$. This implies the proposition.
\end{proof}

\section{Finite-dimensional representations}

A \emph{representation} of a Lie group $G$ is a continuous homomorphism 
$$
\rho:G\to \hbox{GL}_d(\mathbb{C}).
$$
The aim of this section is to explore such representations and, more specifically, find a basis
of $\mathbb{C}^d$ such these representations have the most simple form.
As we shall see, the situation is very different for two classes of groups ---
the solvable groups and the semisimple groups.

We start our discussion with the case of a solvable group.
For a Lie algebra $\mathfrak{g}$, we define inductively
$$
\mathfrak{g}^{(1)}=\left<[x,y]: x,y\in\mathfrak{g}\right>,\;\;\;
\mathfrak{g}^{(2)}=\left<[x,y]: x,y\in\mathfrak{g}^{(1)}\right>, \;\;\; \ldots
$$ 

\begin{definition}
{\rm 
A connected Lie group $G$ is called \emph{solvable} if $\mathcal{L}(G)^{(n)}=0$ for some $n$.
}
\end{definition}

A basic example of a solvable Lie group is any closed subgroup of the group of upper triangular matrices.
The following theorem shows that this example is typical.

\begin{theorem}[Lie-Kolchin]\label{th:lie}
Let $\rho:G\to\hbox{\rm GL}_d(\mathbb{C})$ be a representation of a connected solvable Lie group $G$.
Then there exists $g\in \hbox{\rm GL}_d(\mathbb{C})$ such that $g\rho(G)g^{-1}$ is 
contained in the group of upper triangular matrices.
\end{theorem}

We start the proof with

\begin{lemma}\label{l:gene}
Let $G$ be a connected Lie group.
Then for every nonempty open $U\subset G$, we have $G=\left< U\right>$.
\end{lemma}

\begin{proof}
Let $H=\left<U\right>$. It is clear that $H$ is an open subgroup of $G$.
We have the coset decomposition $G=\sqcup_{g\in G/H} gH$ consisting of disjoint open sets.
Since $G$ is connected, we conclude that $G=H$.
\end{proof}

\begin{proof}[Proof of Theorem \ref{th:lie}]
By Theorem \ref{th:hom}, we have a Lie-algebra homomorphism $D\rho:
\mathcal{L}(G)\to\hbox{M}_d(\mathbb{C})$ such that 
\begin{equation}\label{eq:1}
\rho(\exp(X))=\exp(D\rho(X))\quad\hbox{for all $X\in  \mathcal{L}(G)$.}
\end{equation}
If we prove that $g D\rho(\mathcal{L}(G))g^{-1}$ is of upper triangular form for some $g\in
\hbox{GL}_d(\mathbb{C})$, then it follows from (\ref{eq:1}) that 
$g \rho(U)g^{-1}$ is also of upper triangular form for a neighbourhood $U$ of identity in $G$.
Then it follows from Lemma \ref{l:gene} that $g \rho(G)g^{-1}$ is of upper triangular form as well.
Hence, it remains to show that the Lie algebra $\mathfrak{h}=D\rho(\mathcal{L}(G))$
is upper triangular up to a conjugation. 

We claim that there exists a one-dimensional $\mathfrak{h}$-invariant subspace.
Once this claim is established, the theorem follows by induction on dimension.
Since $\mathfrak{h}$ is solvable, $\mathfrak{h}^{(1)}\ne \mathfrak{h}$.
We take a codimension one subspace $\mathfrak{h}_0$ of $\mathfrak{h}$ that
contains $\mathfrak{h}^{(1)}$ and $X\in \mathfrak{h}$ such that
$\mathfrak{h}=\left<X,\mathfrak{h}_0\right>$.
We note that $[\mathfrak{h}_0,\mathfrak{h}]\subset \mathfrak{h}_0$.
By induction on $\dim(\mathfrak{h})$, there exists a nonzero vector $v$ such that
$\mathfrak{h}_0 v\subset \mathbb{C}v$. For $Y\in \mathfrak{h}_0$, we write
$Yv=\lambda(Y)v$ with $\lambda(Y)\in \mathbb{C}$. Let $v_i=X^iv$. Then
\begin{equation}\label{eq:ind}
Yv_i=YXv_{i-1}=XYv_{i-1}+[Y,X] v_{i-1},
\end{equation}
where $[Y,X]\in \mathfrak{h}_0$.  Using induction on $i$, we deduce from (\ref{eq:ind}) that
the subspaces $\left<v_0,\ldots,v_{i}\right>$ are $\mathfrak{h}_0$-invariant, and moreover,
$$
Yv_i-\lambda(Y)v_i\in \left<v_0,\ldots,v_{i-1}\right>.
$$
Let $V$ be the subspace generated by the vectors $v_i$. In the basis $\{v_i\}$,
the transformation $[Y,X]|_V$ is upper triangular with $\lambda([Y,X])$ on the diagonal.
Hence,
$$
\hbox{Tr}([Y,X]|_V)=\dim(V)\lambda([Y,X]).
$$
On the other hand,
$$
\hbox{Tr}([Y,X]|_V)=\hbox{Tr}(Y|_V X|_V-X|_V Y|_V)=0.
$$
Hence, $\lambda([Y,X])=0$ for every $Y\in \mathfrak{h}_0$.
Then using induction on $i$, we deduce from (\ref{eq:ind})
that $[\mathfrak{h}_0,X]$ acts trivially on $V$, and $Yv_i=\lambda(Y)v_i$.
This proves that every eigenvector of $X$ in $V$ is also an eigenvector of
$\mathfrak{h}=\left<X,\mathfrak{h}_0\right>$, which implies the claim and completes the proof of the theorem.
\end{proof}

\begin{definition}
{\rm
A connected Lie group is called \emph{semisimple} if it contains no nontrivial normal closed
connected solvable subgroups.
}
\end{definition}

An example of a semisimple group is the group $\hbox{SL}_d(\mathbb{R})$.
Representations of semisimple groups behave very differently from representations
of solvable groups.

\begin{theorem}\label{th:weyl}
Let $G$ be a connected semisimple Lie group and $\rho:G\to \hbox{\rm GL}_d(\mathbb{C})$
a representation of $G$. Then every $\rho(G)$-invariant subspace of $\mathbb{C}^d$
has a $\rho(G)$-invariant complement.
\end{theorem}

In particular, we deduce

\begin{corollary}
In the setting of Theorem \ref{th:weyl}, $\mathbb{C}^d=V_1\oplus\cdots\oplus V_s$
where the subspaces $V_i$ are $\rho(G)$-invariant and irreducible (i.e, they don't
contain any proper $\rho(G)$-invariant subspaces).
\end{corollary}

In the proof of the theorem, we use

\begin{lemma}\label{l:invar}
Let $\rho:G\to \hbox{\rm GL}_d(\mathbb{C})$
a representation of a connected Lie group $G$. Then a subspace $V\subset \mathbb{C}^d$
is $\rho(G)$-invariant  if and only if it is $D\rho(\mathcal{L}(G))$-invariant.
\end{lemma}

\begin{proof}
We recall that the following relation holds (see Theorem \ref{th:hom}):
\begin{equation}\label{eq:2}
\rho(\exp(X))=\exp(D\rho(X))\quad\hbox{for all $X\in  \mathcal{L}(G)$.}
\end{equation}

If the subspace $V$ is $G$-invariant, then for every $X\in \mathcal{L}(G)$, $t\in \mathbb{R}$,  and $v\in V$, we have 
$$
\exp(t D\rho(X))v=\rho(\exp(tX))v\in V,
$$
and taking derivative at $t=0$, we obtain that $D\rho(X)v\in V$,
so that $V$ is $D\rho(\mathcal{L}(G))$-invariant.

Conversely, if $V$ is $D\rho(\mathcal{L}(G))$-invariant,
then it follows from (\ref{eq:2}) that it is $\exp(\mathcal{L}(G))$-invariant.
Since by Lemma \ref{l:gene}, $\exp(\mathcal{L}(G))$ generates $G$, this proves the claim. 
\end{proof}

\begin{proof}[Proof of Theorem \ref{th:weyl}]
We give a proof using the so-called ``Weyl's unitary trick''. Surprisingly, the invariant measure
introduced in the previous section turns out to be very useful to prove this algebraic fact.

We first assume that $G$ is compact. Let $\left<\cdot,\cdot\right>$ be a positive-definite
Hermitian form on $\mathbb{C}^d$. We define a new Hermitian form on $\mathbb{C}^d$ by
$$
\left<v_1,v_2\right>_G=\int_G \left<\rho(g^{-1})v_1,\rho(g^{-1})v_2\right>\, dm(g),\quad v_1,v_2\in\mathbb{C}^d,
$$
where $m$ is the left-invariant measure on $G$ constructed in Section 2.
Since $G$ is compact, the measure $m$ is finite,
and the Hermitian form is well-defined. It is also easy see that it is positive-definite.
For $h\in G$ and $v_1,v_2\in\mathbb{C}^d$,
\begin{align*}
\left<\rho(h)v_1,\rho(h)v_2\right>_G&=\int_G \left<\rho(g^{-1}h)v_1,\rho(g^{-1}h)v_2\right>\, dm(h)\\
&=\int_G \left<\rho(g^{-1})v_1,\rho(g^{-1})v_2\right>\, dm(h)=\left<v_1,v_2\right>_G.
\end{align*}
Hence, this form is $\rho(G)$-invariant. 
Given a $\rho(G)$-invariant subspace $V$, we have a decomposition $\mathbb{C}^d=V\oplus V^\perp$,
where $V^\perp=\{v: \left<v,V\right>_G=0\}$. For $v\in V^\perp$,
$$
\left<\rho(g)v,V\right>=\left<v,\rho(g)^{-1}V\right>=0.
$$
This shows that $V^\perp$ is $\rho(G)$-invariant, and proves the theorem in this case.

Now we explain how to give a proof in general. In fact, we restrict our attention to
$G=\hbox{SL}_2(\mathbb{R})$. The same argument works for general groups, but this requires 
some knowledge of the structure theory of semisimple groups, which we don't discuss here.
Let $\mathfrak{g}=\mathcal{L}(G)$ and $D\rho:\mathfrak{g}\to \hbox{M}_d(\mathbb{C})$ 
the corresponding Lie-algebra homomorphism.
We denote by $D\rho_{\mathbb{C}}:\mathfrak{g}\otimes \mathbb{C}\to \hbox{M}_d(\mathbb{C})$
the linear extension of $D\rho$ which is also a Lie-algebra homomorphism.
We consider the subgroup
\begin{align*}
H=\hbox{SU}(2)&=\{g\in \hbox{GL}_2(\mathbb{C}):\, {}^t\bar g g=I,\, \det(g)=1\}\\
&= \left\{ 
\left(
\begin{tabular}{rr}
$a$ & $b$\\
$-\bar b$ & $\bar a$
\end{tabular}
\right):
 \,\, a,b\in \mathbb{C},\, |a|^2+|b|^2=1 \right\}.
\end{align*}
Its Lie algebra is 
\begin{align*}
\mathfrak{h}=\mathcal{L}(H)&=\{X\in \hbox{M}_2(\mathbb{C}):\, {}^t\bar X+ X=0,\, \hbox{Tr}(X)=0\}\\
&= \left\{ 
\left(
\begin{tabular}{rr}
$i u$ & $v$\\
$-\bar v$ & $- i u$
\end{tabular}
\right):
 \,\, u\in \mathbb{R},\, v\in \mathbb{C} \right\}.
\end{align*}
It is easy to check that 
\begin{equation}\label{eq:eq}
\mathfrak{h}\otimes\mathbb{C}=\{X\in \hbox{M}_2(\mathbb{C}):\,
\hbox{Tr}(X)=0\}=\mathfrak{g}\otimes\mathbb{C}.
\end{equation}
Since $H$ is simply connected ($H$ is homeomorphic to the 3-dimensional sphere),
it follows from Theorem \ref{th:simply} that there exists a representation
$\tilde \rho: H\to \hbox{GL}_d(\mathbb{C})$ such that $D\tilde \rho=D\rho_{\mathbb{C}}|_{\mathfrak{h}}$.
Now, if $V$ is a $\rho(G)$-invariant subspace, then by Lemma \ref{l:invar}, it is also 
invariant under $D\rho_{\mathbb{C}}(\mathfrak{g}\otimes
\mathbb{C})=D\rho_{\mathbb{C}}(\mathfrak{h}\otimes \mathbb{C})$, and $\tilde\rho(H)$-invariant.
Since $H$ is compact, we know that $V$ has an $\tilde\rho(H)$-invariant complement $V'$.
Then by Lemma \ref{l:invar}, $V'$ is $D\rho_{\mathbb{C}}(\mathfrak{h})$-invariant.
Hence, it follows from (\ref{eq:eq}) that $V'$ is $D\rho(\mathfrak{g})$-invariant.
Finally, applying Lemma \ref{l:invar} again, we conclude that $V'$ is $\rho(G)$-invariant,
which finishes the proof.

We note that the main ingredient of the proof is existence of a compact subgroup $H$ such that 
(\ref{eq:eq}) holds. Such subgroup is called a compact form of $G$, and it is known that every connected
semisimple Lie group has a compact form. 
\end{proof}

\section{Algebraic groups}

In this section we introduce algebraic groups and discuss their basis properties.

\begin{definition}
{\rm 
A subgroup $G$ of $\hbox{GL}_d(\mathbb{C})$ is called \emph{algebraic}
if it is the zero set of a family of polynomial functions, namely,
$$
G=\{g:\; P(g)=0\;\;\hbox{for all $P\in I$}\}
$$
for some subset $I\subset \mathbb{C}[x_{11},\ldots,x_{dd}]$.
}
\end{definition}

For example, the special linear group $\hbox{SL}_d(\mathbb{C})$ and 
the orthogonal group $\hbox{O}_d(\mathbb{C})$ are algebraic group.
It is clear that every algebraic group can be considered as a Lie group
and results of the previous sections apply. 
The advantage of working with algebraic groups is that they exhibit
much more rigid behaviour than Lie groups. As an example, we mention
the following theorem which will be proved later.

\begin{theorem}\label{th:alg_hom}
Let $f:G_1\to G_2$ be a polynomial homomorphism of algebraic groups $G_1$ and $G_2$.
Then $f(G_1)$ is an algebraic group, in particular, $f(G_1)$ is closed.
\end{theorem}

An analogue of this statement fails in the category of Lie groups.
There are continuous homomorphisms $f:G_1\to G_2$ between Lie groups such that
$f(G_1)$ is not closed. For instance, consider the homomorphism
$$
\mathbb{R}\to \hbox{GL}_2(\mathbb{C}):t\mapsto
\left(
\begin{tabular}{ll}
$e^{2\pi\omega_1 t}$ & 0 \\
$0$ & $e^{2\pi\omega_2 t}$
\end{tabular}
\right),
$$
where $\omega_1,\omega_2\in\mathbb{R}$ are rationally independent.
The image of this map is not closed.

Other examples of rigid behaviour of algebraic groups are absence of nontrivial recurrence points
(Corollary \ref{c:rec} below) and robustness of unipotent and semisimple transformations under
polynomial homomorphisms (Theorem \ref{th:us} below).

For $I\subset \mathbb{C}[x_{1},\ldots,x_{d}]$, we define
$$
\mathcal{V}(I)=\{x\in \mathbb{C}^d:\; P(x)=0\;\;\hbox{for all $P\in I$}\}.
$$
A subset of $\mathbb{C}^d$ is called \emph{algebraic} if it is of the form $\mathcal{V}(I)$ for some $I$.
We list some of the basic properties of the operation $\mathcal{V}$, which are not hard to check:
 
\begin{enumerate}
\item[(i)] $\mathcal{V}(\{1\})=\emptyset$, $\mathcal{V}(\{0\})=\mathbb{C}^d$,
\item[(ii)] $\cap_\alpha \mathcal{V}(I_\alpha)=\mathcal{V}(\cup_\alpha I_\alpha)$,
\item[(iii)] $\mathcal{V}(I_1)\cup \mathcal{V}(I_2)=\mathcal{V}(I_1\cdot I_2)$,
\item[(iv)] If $f:\mathbb{C}^{d_1}\to \mathbb{C}^{d_2}$ is a polynomial map, then
$$
f^{-1}(\mathcal{V}(I))=\mathcal{V}(\{P\circ f:\; P\in I\}).
$$
\end{enumerate}
Properties (i)--(iii) imply that the collection $\{\mathcal{V}(I):\; I\subset \mathbb{C}[x_{1},\ldots,x_{d}]\}$
satisfies the axioms of closed sets and defines a topology on $\mathbb{C}^d$ which is called
the \emph{Zariski topology}. It follows from (iv) that polynomial maps are continuous with
respect to this topology. Although the Zariski topology provides a convenient framework for
studying polynomial maps, the reader should be warned that this topology exhibits many 
counter-intuitive properties. In particular, it is not Hausdorff, and has some compactness
properties (see Proposition 4.4 below).

The usual notion of connectedness is not very useful in this setting
and a natural substitute is the notion of irreducibility:

\begin{definition}
{\rm 
A (Zariski) closed subset $X$ is called \emph{irreducible} if 
$X\ne X_1\cup X_2$ for any closed $X_1, X_2\subsetneq X$.
}
\end{definition}
 
We show that

\begin{proposition}\label{th:irr}
Every closed set $X$ can be decomposed as $X=X_1\cup\cdots \cup X_l$ where $X_i$'s are irreducible closed
sets.
\end{proposition}

In order to prove this theorem, it would be convenient to introduce an operation which is in some sense
the inverse of the map $I\mapsto \mathcal{V}(I)$. For a subset $X\subset \mathbb{C}^d$, we set
$$
\mathcal{I}(X)=\{P\in\mathbb{C}[x_1,\ldots,x_d]:\; P|_X=0\}.
$$ 
It is clear that $\mathcal{I}(X)$ is an ideal in the polynomial ring, and $\mathcal{V}(\mathcal{I}(X))\supset X$. 
In fact, one can check that $\mathcal{V}(\mathcal{I}(X))$ is precisely the closure of $X$
with respect to the Zariski topology.

\begin{proof}[Proof of Proposition \ref{th:irr}]
Suppose that the claim of the proposition is false. Then there exists an infinite  decreasing chain
\begin{equation}\label{eq:chain}
X\supsetneq X_1\supsetneq \cdots \supsetneq X_n\supsetneq\cdots
\end{equation}
where $X_i$'s are closed reducible sets. This gives an increasing chain of ideals
$$
\mathcal{I}(X_1)\subset \cdots \subset \mathcal{I}(X_n)\subset \cdots
$$
in $\mathbb{C}[x_1,\ldots,x_d]$. According to the Hilbert Basis Theorem \cite[Th.~7.5]{am}, every ideal
in $\mathbb{C}[x_1,\ldots,x_d]$ is finitely generated. In particular, the ideal 
$\cup_{n\ge 1} \mathcal{I}(X_n)$ is finitely generated, and it follows that 
$$
\mathcal{I}(X_n)=\mathcal{I}(X_{n+1})=\cdots
$$
for sufficiently large $n$. Since $X_i$'s are closed, $X_i=\mathcal{V}(\mathcal{I}(X_i))$, so that 
the chain (\ref{eq:chain}) stabilises, which is a contradiction.
\end{proof}

The proof of Proposition \ref{th:irr} demonstrates that geometric properties of closed sets
can be studied using tools from Commutative Algebra. This idea turns out to be extremely fruitful.

\begin{definition}
{\rm 
The \emph{coordinate ring} of a closed subset $X$  is defined by
$$
\mathcal{A}(X)=\mathbb{C}[x_1,\ldots,x_d]/\mathcal{I}(X)
$$
}
\end{definition}

Many geometric properties can reformulated in the language of Commutative Algebra,
as demonstrated by Table \ref{t:alg} below. 

\begin{table}[h]\label{t:alg}
\begin{tabular}{|c|c|}
\hline
& \\
{\bf Geometry} & {\bf Commutative Algebra}\\
& \\
\hline\hline
& \\
points in $X$ & algebra homomorphisms $\mathcal{A}(X)\to\mathbb{C}$\\
&\\
\hline
& \\
$X$ is irreducible & $\mathcal{A}(X)$ has no divisors of zero\\
& \\
\hline
& \\
polynomial maps $f:X\to Y$ & algebra homomorphisms $f_*: \mathcal{A}(Y)\to \mathcal{A}(X)$\\
& \\
\hline
& \\
$\overline{f(X)}=Y$ & $f_*$ is injective\\
& \\
\hline
\end{tabular}
\caption{Algebraic correspondence}
\end{table}

To check the first line of Table \ref{t:alg}, we observe that any $a\in\mathbb{C}^d$ defines
an algebra homomorphism
$$
\alpha_{a}: \mathbb{C}[x_1,\ldots,x_d]\to \mathbb{C}:P\mapsto P(a).
$$
Moreover, if $a\in X$, then $\mathcal{I}(X)\subset \ker(\alpha_{a})$ and
$\alpha_{a}$ defines a homomorphism $\mathcal{A}(X)\to \mathbb{C}$.
Conversely, any homomorphism $\mathcal{A}(X)\to \mathbb{C}$ is of the form
$P\mapsto P(a)$, where $Q(a)=0$ for all $Q\in \mathcal{I}(X)$, i.e., $a\in X$.

In regard to the third line, we note that any polynomial map $f:\mathbb{C}^{d_1}\to \mathbb{C}^{d_2}$
defines an algebra homomorphism 
$$
f_*:\mathbb{C}[y_1,\ldots,y_{d_2}]\to \mathbb{C}[x_1,\ldots,x_{d_1}]:P\mapsto P\circ f.
$$
This homomorphism defines a map $\mathbb{C}[y_1,\ldots,y_{d_2}]/\mathcal{I}(Y)\to
\mathbb{C}[x_1,\ldots,x_{d_1}]/\mathcal{I}(X)$ if 
$$
f_*(\mathcal{I}(Y))\subset \mathcal{I}(X)\;\; \Longleftrightarrow\;\; \forall P\in \mathcal{I}(Y):\, P|_{f(X)}=0\;\;
\Longleftrightarrow \;\; f(X)\subset Y.
$$
Conversely, any homomorphism $\mathcal{A}(Y)\to \mathcal{A}(X)$ is of this form.

To check the fourth property in Table 1, we observe that $\overline{f(X)}=Y$ is equivalent to 
$$
\forall P\in \mathcal{A}(Y)\backslash\{0\}:\, P|_{f(X)}\ne 0
\;\;\Longleftrightarrow\;\; 
\forall P\in \mathcal{A}(Y)\backslash\{0\}:\,
f_*(P)\notin \mathcal{I}(X),
$$
which means that $f_*:\mathcal{A}(Y)\to \mathcal{A}(X)$ is injective.

The following proposition will be used in the proof of Theorem \ref{th:alg_hom}.

\begin{proposition}\label{p:open}
Let $f:X\to Y$ be a polynomial map between closed sets $X$ and $Y$.
Then $f(X)$ contains an open subset of $\overline{f(X)}$.
\end{proposition}

We note that for this proposition it is crucial that the field $\mathbb{C}$ is algebraically closed,
and the analogous statement fails for polynomial maps $\mathbb{R}^{d_1}\to \mathbb{R}^{d_2}$.

\begin{proof}
Using Proposition \ref{th:irr}, we may reduce the proof to the case when $X$ is irreducible,
and without loss of generality we may assume that $Y=\overline{f(X)}$.
Then we have an injective algebra homomorphism $f_*:\mathcal{A}(Y)\to \mathcal{A}(X)$.
We claim that
\begin{equation}\label{eq:cc}
\exists P\in\mathcal{A}(Y):\, \{P\ne 0\}\cap Y\subset f(X).
\end{equation}
Since $\{P\ne 0\}\cap Y$ is open in $Y$, this implies the proposition.
We use the correspondence:
$$
\begin{tabular}{ccc}
$\{\hbox{points in $Y$}\}$ & $\longleftrightarrow$ & $\{\hbox{homomorphisms $\mathcal{A}(Y)\to\mathbb{C}$}\}$\\
$\bigcup$ && $\bigcup$\\
$\{\hbox{points in $f(X)$}\}$ & $\longleftrightarrow$ & $\{\hbox{homomorphisms factoring through $f_*$}\}$.
\end{tabular}
$$
Let $A=f_*(\mathcal{A}(Y))$ and $B=\mathcal{A}(X)$. The claim (\ref{eq:cc}) is equivalent to showing that
there exists $Q\in A$ such that every homomorphism $\phi:A\to \mathbb{C}$ with $\phi(Q)\ne 0$ extends to
a homomorphism $B\to \mathbb{C}$. This statement is proved, for instance, in \cite[Prop.~5.23]{am}. 
\end{proof}

Now we finally deduce Theorem \ref{th:alg_hom} from Proposition \ref{p:open}.

\begin{proof}[Proof of Theorem \ref{th:alg_hom}]
We consider the subgroup $L=f(G_1)$. Then its closure $\overline{L}$ is also a subgroup.
Indeed, since the multiplication and inverse operations are continuous in Zariski topology,
$$
\overline{L}{}^{-1} \cdot \overline{L}\subset \overline{L^{-1}\cdot L}\subset \overline{L}.
$$
By Proposition \ref{p:open}, $L$ contains an open subset of $\overline{L}$, and it follows that
$L$ is an open subgroup of $\overline{L}$. We have the coset decomposition 
$$
\overline{L}=\bigsqcup_{l\in
  \overline{L}/L} l L,
$$
where each of the cosets is open in $\overline{L}$. Hence, $L$ is closed,
which completes the proof.
\end{proof}

\begin{definition}
{\rm A closed subset $X\subset \mathbb{C}^d$ is \emph{defined over $K$}
(for a subfield $K$ of $\mathbb{C}$) if the ideal $\mathcal{I}(X)$ is generated by elements
in $K[x_1,\ldots,x_d]$.
}
\end{definition}

For a closed subset defined over $K$, we set $X(K)=X\cap K^d$.
In general, the set $X(K)$ could be quite small and even empty, but in the setting
 of algebraic groups, we have:

\begin{proposition}
Let $G$ be an algebraic group defined over $\mathbb{R}$.  
Then $G(\mathbb{R})$ is a Lie group of dimension $\dim_{\mathbb{C}} (G)$.
\end{proposition}

%There is also a version of this proposition for general fields which we don't discuss here.

\begin{proof}
Suppose that the group $G$ is defined by a system $P_1=\cdots =P_s=0$ of polynomial equations
with real coefficients. We recall from Remark \ref{r:lie} that the Lie algebra can be computed
as the tangent space at identity, so that
\begin{align*}
\mathcal{L}(G)&=\{X\in\hbox{M}_d(\mathbb{C}):\, (DP_1)_I X=\cdots=(DP_s)_I X=0\},\\
\mathcal{L}(G(\mathbb{R}))&=\{X\in\hbox{M}_d(\mathbb{R}):\, (DP_1)_I X=\cdots=(DP_s)_I X=0\}.
\end{align*}
Since 
$$
\dim_{\mathbb{C}}(\mathcal{L}(G))=\dim_{\mathbb{R}}(\mathcal{L}(G(\mathbb{R}))),
$$
the claim follows.
\end{proof}

The following result is one of the main theorems of this section, which shows that orbits for polynomial
actions behave nicely.

\begin{theorem}\label{th:orbit0} 
Let $G$ be an irreducible algebraic group defined over $\mathbb{R}$,
$X\subset \mathbb{C}^d$ a Zariski closed set defined over $\mathbb{R}$,
$x\in X(\mathbb{R})$, and $G\times X\to X$ a polynomial action defined over $\mathbb{R}$.
We denote by $Y$ the Zariski closure of $G\cdot x$ in $X$.
Then the map
$$
G(\mathbb{R})\to Y(\mathbb{R}):g\mapsto g\cdot x
$$
is open with respect to the Euclidean topology.
\end{theorem}

\begin{proof}
Without loss of generality, we may assume that $X=Y$. Then since $G$ is irreducible, $Y$ is irreducible.
For Proposition \ref{p:open} we know that $G\cdot x$ contains a Zariski open subset $X$.
Since $G$ acts transitively on $G\cdot x$, it follows that $G\cdot x$ is, in fact, Zariski open in $X$.

Let $X_0$ be the set of smooth points of $X$ (i.e., the set of points where the tangent space has minimal
dimension). This set is Zariski open in $X$ and $G$-invariant. Since $X$ is irreducible, the intersection
of finitely many nonempty Zariski open subsets in $X$ is nonempty. In particular,
$G\cdot x\cap X_0\ne \emptyset$, and it follows that $G\cdot x\subset X_0$.

We consider the orbit map $F:G\to X: g\mapsto g\cdot x$ and its derivative
$(DF)_g: T_g(G)\to T_{g\cdot x}(X)$, where  $T_g(G)$ and $T_{g\cdot x}(X)$
denote the corresponding tangent spaces. Since $G\cdot x$ is Zariski open in $X$,
the map $(DF)_g$ is onto. Then the map $(DF)_g: T_g(G(\mathbb{R}))\to T_{g\cdot x}(X(\mathbb{R}))$
is also onto. Hence, by the Implicit Function Theorem, the map $F:G(\mathbb{R})\to X(\mathbb{R})$
is open with respect to the Euclidean topology, as required.
\end{proof}

\begin{definition}
{\rm
Let $\{s(t)\}_{t\in \mathbb{R}}$ be a one-parameter group acting on a topological space $X$.
A point $x\in X$ is called \emph{recurrent} if $s(t_n)\cdot x\to x$ for some sequence $t_n\to \infty$.
}
\end{definition}

Using Theorem \ref{th:orbit}, we obtain a complete description of recurrent points for algebraic actions.

\begin{corollary}\label{c:rec}
Let $S=\{s(t)\}_{t\in\mathbb{C}}$ be a one-dimensional algebraic group defined over $\mathbb{R}$,
$X\subset \mathbb{C}^d$ a Zariski closed set defined over $\mathbb{R}$,
and $S\times X\to X$ a polynomial action defined over $\mathbb{R}$.
Then all $S(\mathbb{R})$-recurrent points in $X(\mathbb{R})$ are fixed by $S$.
\end{corollary}

\begin{proof}
By Theorem \ref{th:orbit0}, the set $s((-\epsilon,\epsilon))\cdot x$ is open in
$\overline{S(\mathbb{R})\cdot x}$. 
Hence, if $s(t_n)\cdot x\to x$, then $s(t_n)\cdot x\in  s((-\epsilon,\epsilon))\cdot x$
for all sufficiently large $n$. This implies that $\hbox{Stab}_S(x)$ is infinite.
Since $S$ is one-dimensional, $\hbox{Stab}_S(x)$ is Zariski dense in $S$.
On the other hand, it is clear that $\hbox{Stab}_S(x)$ is Zariski closed.
Thus, $\hbox{Stab}_S(x)=S$, as claimed.
\end{proof}

We complete this section with discussion of semisimple and unipotent elements.

\begin{definition}
{\rm
\begin{itemize}
\item An element $g\in \hbox{GL}_d(\mathbb{C})$ is called \emph{semisimple} if it is diagonalisable over
  $\mathbb{C}$.
\item An element $g\in \hbox{GL}_d(\mathbb{C})$ is called \emph{unipotent} if all of its eigenvalues of
  $g$
are equal to one.
\end{itemize}
}
\end{definition}

We note that it follows from the Jordan Canonical Form that every  element $g\in \hbox{GL}_d(\mathbb{C})$
can written as $g=g_s g_u$ where $g_s$ and $g_u$ are commuting semisimple and unipotent elements.

\begin{theorem}\label{th:us}
Let $\rho:\hbox{\rm GL}_d(\mathbb{C})\to\hbox{\rm GL}_N(\mathbb{C})$ be a polynomial homomorphism. Then
\begin{itemize}
\item if $g\in \hbox{\rm GL}_d(\mathbb{C})$ is semisimple,  $\rho(g)$ is also semisimple,
\item if $g\in \hbox{\rm GL}_d(\mathbb{C})$ is unipotent, $\rho(g)$ is also unipotent.
\end{itemize}
\end{theorem}

\begin{proof}
Suppose that $g$ is semisimple. 
Let $V_\lambda\subset\mathbb{C}^N$ be a Jordan subspace of $\rho(g)$ with the eigenvalue $\lambda$.
Then the linear map $\lambda^{-n}\rho(g)^n|_{V_\lambda}$ has coordinates which are polynomials in $n$.
On the other hand, these coordinates can be expressed as polynomials in $\lambda^{-n}$,
$\lambda^n_1,\cdots,\lambda^n_s$ where $\lambda_i$'s are the eigenvalues of $g$. This implies that
all these coordinates are constant, and $\lambda^{-n}\rho(g)^n|_{V_\lambda}=1$. Hence, $\rho(g)$ is semisimple.

Suppose that $g$ is unipotent. Let $v\in \mathbb{C}^N$ be an eigenvector of $\rho(g)$ with eigenvalue
$\lambda$. Then $\rho(g^n)v=\lambda^n v$, but $\rho(g^n)v$ has coordinates which are polynomials in $n$.
This implies that $\lambda=1$. Hence, $\rho(g)$ is unipotent.
\end{proof}

\section{Lattices -- geometric constructions}

A linear flow on the torus $\mathbb{T}^d=\mathbb{R}^d/\mathbb{Z}^d$ 
is one of the most basic examples of dynamical systems.
More generally, one may consider a factor space $\Gamma\backslash G$, where
$G$ is a Lie group and $\Gamma$ is a discrete subgroup, and define 
a flow on $X$ acting by a one-parameter subgroup of $G$.
In some cases the space $\Gamma\backslash G$ can be equipped with a finite invariant measure.
This construction provides a rich and very important family of dynamical systems.
Besides the theory of dynamical systems, such spaces also play important role in
geometry and number theory.

In this section, we cover basic material regarding the factor spaces $\Gamma\backslash G$.
In  particular, we define a measure on $\Gamma\backslash G$, which
is induced by the invariant measure on $G$, and explain the Poincare's geometric construction
of discrete cocompact subgroup $\Gamma$ in $\hbox{SL}_2(\mathbb{R})$.

Let $G$ be a Lie group and $\Gamma$ a discrete subgroup of $G$.

\begin{definition}
{\rm
A subset $F\subset G$ is called a \emph{fundamental set} for $\Gamma$ if 
$G$ is equal to the disjoint union of the sets $\gamma F$, $\gamma\in\Gamma$:
$$
G=\bigsqcup_{\gamma\in \Gamma} \gamma F.
$$
}
\end{definition}

For example, $F=[0,1)^d$ is a fundamental set of $\mathbb{Z}^d\subset\mathbb{R}^d$.

\begin{lemma}
There exists a Borel fundamental set of $\Gamma$.
\end{lemma} 

\begin{proof}
Since $\Gamma$ is a discrete subgroup of $G$, there exists a neighbourhood $U$
of identity in $G$ such that 
\begin{equation}\label{eq:e1}
\Gamma\cap U\cdot U^{-1}=\{I\}.
\end{equation}
We can write 
\begin{equation}\label{eq:e2}
G=\bigcup_{n=1}^\infty Ug_n
\end{equation}
for a sequence $g_n\in G$. Let
$$
F=\bigcup_{n=1}^\infty \left(Ug_n\backslash (\cup_{i=1}^{n-1} \Gamma Ug_i)\right).
$$
It follows from (\ref{eq:e2}) that $G=\Gamma F$, and using (\ref{eq:e1}), it is easy to deduce that
if $\gamma_1 F\cap\gamma_2 F\ne \emptyset$, then $\gamma_1=\gamma_2$. Hence, $F$ is a fundamental set for $\Gamma$. 
\end{proof}

We denote by $m$ the left-invariant measure on $G$ constructed in Section 2 and 
by $\pi:G\to \Gamma\backslash G$ the factor map. Taking a Borel fundamental domain $F$ for $\Gamma$,
we define a measure on $\Gamma\backslash G$ by 
$$
\mu(B)=m(\pi^{-1}(B)\cap F)\quad\hbox{for all Borel $B\subset \Gamma\backslash G$.}
$$

\begin{lemma}
\begin{enumerate}
\item[(i)] The definition of $\mu$ does not depend on a choice of the fundamental domain $F$.
\item[(ii)] If $m(F)<\infty$, then the measure $\mu$ is right $G$-invariant.
\end{enumerate}
\end{lemma}

\begin{proof}
Let $F_1,F_2\subset G$ be Borel fundamental sets for $\Gamma$.
Since
$$
G=\bigsqcup_{\gamma\in \Gamma} \gamma F_1=\bigsqcup_{\gamma\in \Gamma} \gamma F_2,
$$
we obtain using left-invariance of $m$ that for every Borel $B\subset \Gamma\backslash G$,
\begin{align*}
m(\pi^{-1}(B)\cap F_1)&=\sum_{\gamma\in\Gamma} m(\pi^{-1}(B)\cap F_1\cap \gamma F_2)\\
&=\sum_{\gamma\in\Gamma} m(\pi^{-1}(B)\cap \gamma^{-1}F_1\cap  F_2)=m(\pi^{-1}(B)\cap F_2).
\end{align*}
This proves (i).

To prove (ii), we consider the measures $m_g$, $g\in G$, on $G$ defined by
$$
m_g(A)=m(Ag)\quad\hbox{for all Borel $A\subset G$.}
$$
It is clear that $m_g$ is left-invariant and locally finite, so that by Theorem \ref{th:unique},
$m_g=c_g\cdot m$ for some $c_g>0$. Since $Fg^{-1}$ is also a fundamental domain for $\Gamma$,
for every Borel $B\subset \Gamma\backslash G$,
\begin{align}\label{eq:c_g}
m(\pi^{-1}(Bg)\cap F)&=m_g(\pi^{-1}(B)\cap Fg^{-1})=c_g\, m(\pi^{-1}(B)\cap Fg^{-1})\\
&=c_g\, m(\pi^{-1}(B)\cap F).  \nonumber
\end{align}
This shows that
$$
\mu(B g)=c_g\, \mu(B)\quad\hbox{for all Borel $B\subset \Gamma\backslash G$.}
$$
It follows from (\ref{eq:c_g}) that $m(F)=c_g\, m(F)$. Hence, if $m(F)<\infty$, then
$c_g=1$, and the measure $\mu$ is $G$-invariant.
\end{proof}

\begin{definition}
{\rm 
A discrete subgroup $\Gamma$ of a Lie group $G$ is called a \emph{lattice} if
$\mu(\Gamma\backslash G)<\infty$.
}
\end{definition}

One can show that any lattice in $\mathbb{R}^d$ is of the form $\mathbb{Z}^dv_1+\cdots+\mathbb{Z}^dv_d$,
where $v_1,\ldots,v_d$ is a basis of $\mathbb{R}^d$. The situation is much more interesting for 
lattices in $\hbox{SL}_2(\mathbb{R})$, and in the rest of this section we construct some examples
of such lattices. 

\vspace{0.5cm}

We introduce the upper-half model of the hyperbolic plane. Let
$$
\mathbb{H}=\{x+iy:\, x\in\mathbb{R}, y>0\}.
$$
The (hyperbolic) length of a $C^1$ curve $c:[0,1]\to \mathbb{H}$ is defined by
$$
L(c)=\int_0^1 \frac{\|c'(t)\|}{\hbox{\small Im}(c(t))}\, dt.
$$
For $g=\left( 
\begin{tabular}{ll}
$a$ & $b$\\
$c$ & $d$
\end{tabular}
\right)\in \hbox{SL}_2(\mathbb{R})$, we define 
$$
T_g:\mathbb{H}\to \mathbb{H}: z\mapsto \frac{az+b}{cz+d}.
$$
The following properties are easy to check:
\begin{enumerate}
\item[(i)] $T_g=id$ if and only if $g=\pm I$,
\item[(ii)] $T_{g_1}T_{g_2}=T_{g_1g_2}$,
\item[(iii)] $\hbox{Im}(T_g(z))=\frac{\hbox{\small Im}(z)}{|cz+d|^2}$,
\item[(iv)] $\hbox{Stab}_{\hbox{\tiny SL}_2(\mathbb{R})}(i)=\hbox{SO}_2(\mathbb{R})$,
\item[(v)] $T_g$ preserves length and angles between curves.
\end{enumerate}
Note that (iii) implies that $T_g(\mathbb{H})\subset\mathbb{H}$.
Let 
$$
u(x)=\left( 
\begin{tabular}{ll}
$1$ & $x$\\
$0$ & $1$
\end{tabular}
\right)\quad\hbox{and}\quad a(y)=\left( 
\begin{tabular}{ll}
$y^{1/2}$ & $0$\\
$0$ & $y^{-1/2}$
\end{tabular}
\right).
$$
Then 
\begin{equation}\label{eq:T}
T_{u(x)a(y)}(i)=x+iy.
\end{equation}
This shows that $\hbox{SL}_2(\mathbb{R})$
acts transitively on $\mathbb{H}$, and by (iv),
$$
\mathbb{H}\simeq \hbox{SL}_2(\mathbb{R})/\hbox{SO}_2(\mathbb{R}).
$$
Moreover, we deduce the Iwasawa decomposition:
\begin{equation}\label{eq:iwasawa}
\hbox{SL}_2(\mathbb{R})=\{u(x)a(y)k:\, x\in \mathbb{R}, y>0, k\in\hbox{SO}_2(\mathbb{R})\}.
\end{equation}

Now we identify the shortest paths in $\mathbb{H}$.

\begin{lemma}
The geodesic (i.e., the shortest path) between $z_1,z_2\in\mathbb{H}$ is either
a vertical line of a semi-circle with the centre on the $x$-axis.
\end{lemma}

\begin{proof}
We first consider the case when $\hbox{Re}(z_1)=\hbox{Re}(z_2)$.
Given a path $c:[0,1]\to \mathbb{H}$ between $z_1$ and $z_2$, we have an estimate
$$
L(c)=\int_0^1 \frac{\sqrt{c_1'(t)^2+c_2'(t)^2}}{c_2(t)}\, dt \ge \int_0^1 \frac{|c'_2(t)|}{c_2(t)}\, dt,
$$
where the equality holds when $c_1'=0$. This implies that the shortest path is a vertical line.

In general, given $z_1,z_2\in\mathbb{H}$, one can find $g\in \hbox{SL}_2(\mathbb{R})$
such that 
$$
\hbox{Re}(T_g(z_1))=\hbox{Re}(T_g(z_2))=0.
$$
Then it follows from the property (v)
that the shortest between $z_1$ and $z_2$ is the image of the $y$-axis under the transformation
$T_g^{-1}$. It can be computed directly that this image is either a vertical line or a semi-circle.
\end{proof}

Besides the transformations $T_g$, we also introduce \emph{reflexion} maps $R_\ell$ with respect to 
a geodesic $\ell$. Given $z\in \mathbb{H}$, we draw a geodesic through $z$ which is orthogonal to
$\ell$ and define $R_\ell(z)$ as the reflection with respect to the intersection point.
More explicitly, if $\ell_0$ is the $y$-axis, then $R_{\ell_0}:z\mapsto -\bar z$,
and in general $R_\ell=T_g^{-1}R_{\ell_0}T_g$, where $g\in \hbox{SL}_2(\mathbb{R})$ is such that
$T_g(\ell)=\ell_0$. We note that the transformations $R_\ell$ also preserve length and angles between curves,
and the group generated by the transformations $T_g$ and $R_\ell$ is an index two supergroup of 
$T_{\hbox{\tiny SL}_2(\mathbb{R})}$.

\vspace{0.5cm}

Now we are ready to construct a family of cocompact lattices in $\hbox{SL}_2(\mathbb{R})$.
One can check that for every $\alpha,\beta,\gamma>0$ such that $\alpha+\beta+\gamma<\pi$
there exists a geodesic triangle with angles $\alpha,\beta,\gamma$.
We fix a triangle $\mathcal{T}$ with angles $\frac{\pi}{n_1},\frac{\pi}{n_2},\frac{\pi}{n_3}$ where $n_i$'s
are integers, and denote by $R_1,R_2,R_3$ the reflections
with respect to the sides of this triangle. Let $\Lambda$ be the group generated by 
these transformations. For every $\lambda\in\Lambda$, $\lambda\mathcal{T}$ is another triangle
with the same dimensions. Since $n_1,n_2,n_3$ are integers, the images of $\mathcal{T}$ fit together
perfectly around every vertex. Hence, we obtain the tiling
\begin{equation}\label{eq:tiling}
\mathbb{H}=\bigcup_{\lambda\in\Lambda} \lambda \mathcal{T},
\end{equation}
and if $\lambda_1 \mathcal{T}^\circ\cap \lambda_2 \mathcal{T}^\circ\ne \emptyset$, then
$\lambda_1 \mathcal{T}^\circ=\lambda_2 \mathcal{T}^\circ$ and $\lambda_1=\lambda_2$.
Let $\Lambda_0$ be the subgroup of $\Lambda$ of index two consisting of elements which are products
of even number of reflections. Then $\Lambda_0\subset T_{\hbox{\tiny SL}_2(\mathbb{R})}$. We set
$$
\Gamma=T^{-1}(\Lambda_0)\subset \hbox{SL}_2(\mathbb{R}).
$$

\begin{theorem}
The group $\Gamma$ is a cocompact lattice in $\hbox{\rm SL}_2(\mathbb{R})$.
\end{theorem}

\begin{proof}
We consider the map
$$
p:\hbox{SL}_2(\mathbb{R})\to\mathbb{H}: g\mapsto T_g(i).
$$
It satisfies the equivariance property
$$
p(gh)=T_g(p(h))\quad\hbox{ for all $g,h\in\mathbb{H}$.}
$$
Using (\ref{eq:T}) and (\ref{eq:iwasawa}), it is easy to deduce that 
this map is proper, so that
$$
F=p^{-1}(\mathcal{T}\cup R_1(\mathcal{T}))
$$
is compact. Since 
$$
\mathbb{H}=\bigcup_{\lambda\in\Lambda_0} \lambda(\mathcal{T}\cup R_1(\mathcal{T})),
$$
we conclude that $G=\Gamma F$. Hence, $\Gamma$ is cocompact.

To prove that $\Gamma$ is discrete, we observe that every compact subset in $\mathbb{H}$
is covered by finitely many tiles in (\ref{eq:tiling}). This implies that for every 
compact $\Omega\subset \mathbb{H}$,
$$
|\{\lambda\in \Lambda:\, \lambda\Omega\cap \Omega\}|<\infty.
$$
Therefore, for any compact $\tilde\Omega \subset\hbox{SL}_2(\mathbb{R})$,
$$
|\{\gamma\in \Gamma:\, \gamma\tilde \Omega\cap \tilde \Omega\}|<\infty.
$$
It follows that every compact subset of $\hbox{SL}_2(\mathbb{R})$ contains only finitely
many elements of $\Gamma$, so that $\Gamma$ is discrete.
\end{proof}

\section{Lattices -- arithmetic constructions} 

In this section we discuss arithmetic constructions of lattices in Lie groups
beginning with the most basic example:

\begin{theorem}\label{th:sl}
$\hbox{\rm SL}_d(\mathbb{Z})$ is a lattice in $\hbox{\rm SL}_d(\mathbb{R})$.
\end{theorem}

In order to prove this theorem, it would be convenient to identify the factor space
$\hbox{SL}_d(\mathbb{Z})\backslash\hbox{SL}_d(\mathbb{R})$ with the space of unimodular lattices in $\mathbb{R}^d$.
A lattice in $\mathbb{R}^d$ is subgroup of the form $\mathbb{Z}v_1+\cdots+\mathbb{Z}v_d$
where $\{v_1,\ldots,v_d\}$ is a basis of $\mathbb{R}^d$. We denote by $\mathcal{L}_d$ the
set of lattices in $\mathbb{R}^d$ with covolume one. 
It would be convenient to write elements of $\mathbb{R}^d$ as row vectors.

We observe that the group
$\hbox{SL}_d(\mathbb{R})$ naturally acts on $\mathcal{L}_d$:
$$
L\mapsto Lg,\quad\quad L\in \mathcal{L}_d,\; g\in G.
$$
This action is transitive and the stabiliser of the lattice $\mathbb{Z}^d$ is $\hbox{SL}_d(\mathbb{Z})$,
so that we have the identification
$$
\mathcal{L}_d\simeq\hbox{SL}_d(\mathbb{Z})\backslash\hbox{SL}_d(\mathbb{R}).
$$

We introduce the \emph{Iwasawa decomposition} for $\hbox{SL}_d(\mathbb{R})$, which is a generalisation of (\ref{eq:iwasawa}).

\begin{lemma}\label{l:iwasawa}
$$
\hbox{\rm SL}_d(\mathbb{R})=NAK,
$$
where
\begin{align*}
N&=\hbox{ the unipotent upper triangular group},\\
A&=\left\{\hbox{\rm diag}(a_1,\ldots,a_d):\,\, a_i>0,\, \prod_{i=1}^d a_i=1\right\},\\
K&=\hbox{\rm SO}_d(\mathbb{R}).
\end{align*}
\end{lemma}

Lemma \ref{l:iwasawa} is easy to proved using the Gramm--Schmidt orthonormalisation process.

In order to prove Theorem \ref{th:sl}, it is sufficient to construct
a set $\Sigma\subset \hbox{\rm SL}_d(\mathbb{R})$
such that $\hbox{\rm SL}_d(\mathbb{R})=\hbox{\rm SL}_d(\mathbb{Z})\Sigma$ and $m(\Sigma)<\infty$,
where $m$ is the invariant measure on $\hbox{\rm SL}_d(\mathbb{R})$.
Hence, Theorem \ref{th:sl} would follow from Lemmas \ref{l:cover} and \ref{l:mm} below.
For this purpose, we introduce the \emph{Siegel sets}:
$$
\Sigma_{s,t}=\{nak:\,\, |n_{ij}|\le s,\, a_i/a_{i-1}\le t,\, k\in\hbox{SO}_d(\mathbb{R})\}.
$$

\begin{lemma}\label{l:cover}
For $s\ge 1/2$ and $t\ge 2/\sqrt{3}$, 
$$
\hbox{\rm SL}_d(\mathbb{R})=\hbox{\rm SL}_d(\mathbb{Z})\Sigma_{s,t}.
$$
\end{lemma}

\begin{proof}
Let $g\in \hbox{\rm SL}_d(\mathbb{R})$ and $L=\mathbb{Z}^dg$. We would like to find a basis
of the lattice $L$ with the ``least complexity''. We say that a basis $(v_1,\ldots,v_d)$ of $L$ is \emph{reduced} if
\begin{enumerate}
\item[(i)] The vector $v_d$ has the smallest norm in $L\backslash \{0\}$,
\item[(ii)] Let $P:\mathbb{R}^d\to \left<v_d\right>^\perp$ denote the orthogonal projection.
The tuple $(P(v_1),\ldots,P(v_{d-1}))$ is a reduced basis of $P(L)$.
\item[(iii)] The vectors $P(v_i)$, $i\ge 2$, have the minimal norms in $P^{-1}(P(v_i))$.
\end{enumerate}
Using induction on $d$, one can show that a reduced basis always exists.

Let $(v_1,\ldots,v_d)$ be a reduced basis of $L$. There exists $h\in\hbox{SL}_d(\mathbb{R})$
such that $v_i=e_i h$, where $(e_1,\ldots,e_d)$ is the standard basis.
Since $\mathbb{Z}^dg=\mathbb{Z}^dh$, we have $g\in \hbox{SL}_d(\mathbb{Z}) h$. Hence,  
it is sufficient to show that $h$ belong to the Siegel set $\Sigma_{s,t}$.

We decompose $h$ as $h=nak$ with $n\in N$, $a\in A$, and $k\in K$. Let 
$$
w_i=v_ik^{-1}=e_ina=a_i e_i +a_{i+1}n_{i,i+1}e_{i+1}+\cdots +a_d n_{i,d}e_d.
$$
Since $k$ preserve the Euclidean product, $(w_1,\ldots,w_d)$ is a reduced
basis for the lattice $Lk^{-1}$. We claim that 
\begin{equation}\label{eq:ind0}
|n_{ij}|\le 1/2\quad\hbox{and}\quad a_i/a_{i-1}\le 2/\sqrt{3}.
\end{equation}
Because of property (ii), we may assume inductively that (\ref{eq:ind0}) holds for $i,j\le d-1$.
Property (iii) implies that for all $\ell \in \mathbb{Z}$,
$$
\|w_i\|=\sqrt{a_i^2+\cdots+a_d^2n_{id}^2}\ge \|w_i+\ell w_d\|=\sqrt{a_i^2+\cdots+a_d^2(n_{id}+\ell)^2}.
$$
This implies that $|n_{id}|\le 1/2$. By property (i),
$$
\|w_d\|=a_d\le \|w_{d-1}\|=\sqrt{a_{d-1}^2+a_d^2n_{d-1,d}^2}.
$$ 
Hence, 
$$
a_d^2\le a_{d-1}^2+a_d^2/4,
$$
and $a_d/a_{d-1}\le 2/\sqrt{3}$. This proves (\ref{eq:ind0}) and completes the proof of the lemma.
\end{proof}

Using the Iwasawa decomposition, we deduce a convenient formula for the left-invariant measure $m$
on $\hbox{SL}_d(\mathbb{R})$ using the coordinates $n_{ij}$, $i<j$, $b_i=a_{i}/a_{i-1}$, $i=2,\ldots,d$, $k\in K$
with respect to the Iwasawa decomposition.

\begin{lemma}\label{l:measure}
The left-invariant measure $m$ on $\hbox{\rm SL}_d(\mathbb{R})$ is given by
$$
\int_{\hbox{\small \rm SL}_d(\mathbb{R})} f\, dm= \int_{N\times A\times K} f(nak)\, \left(\prod_{i<j}dn_{ij}\right)
\left(\prod_{i=2}^d b_i^{r_i}\,db_i\right) d\nu(k),
$$
where $r_i\in\mathbb{N}$ and $\nu$ is the (finite) right-invariant measure on $K$.  
\end{lemma}

\begin{proof}
By Lemma \ref{l:iwasawa}, $\hbox{\rm SL}_d(\mathbb{R})$ is a product of the groups $NA$ and $K$.
The left-invariant measure on $NA$ can be computed explicitly. Then the lemma follows from
Proposition \ref{p:dec}.
\end{proof}

The following lemma can be checked by a direct computation.

\begin{lemma}\label{l:mm}
$m(\Sigma_{s,t})<\infty$.
\end{lemma}

Now Theorem \ref{th:sl} follows from Lemmas \ref{l:cover} and \ref{l:mm}.

\vspace{0.5cm}

The space $\mathcal{L}_d\simeq \hbox{SL}_d(\mathbb{Z})\backslash\hbox{SL}_d(\mathbb{R})$ is equipped with
the factor topology defined by the map
$\hbox{SL}_d(\mathbb{R})\to\hbox{SL}_d(\mathbb{Z})\backslash\hbox{SL}_d(\mathbb{R})$. 
A sequence of lattices $L^{(n)}$ converges to $L$ if there exist bases $\{v_i^{(n)}\}$
of $L^{(n)}$ that converge to a basis of $L$. We observe that the space $\mathcal{L}_d$
is not compact. Indeed, it is clear that the sequence of lattices $\mathbb{Z}(\frac{1}{n}e_1)+\mathbb{Z}(n
e_2)+\mathbb{Z} e_3+\cdots+\mathbb{Z} e_d$ has no convergence subsequences.
The following theorem provides a convenient compactness criterion.

\begin{theorem}[Mahler compactness criterion]\label{th:mahler}
A subset $\Omega\subset \mathcal{L}_d$ is relatively compact if and only if there exists $\delta>0$
such that 
\begin{equation}\label{eq:compact}
\|v\|\ge \delta\quad\hbox{for every $L\in\Omega$ and $v\in L\backslash \{0\}$}.
\end{equation}
\end{theorem}

\begin{proof}
Suppose that (\ref{eq:compact}) holds.
It follows from Lemma \ref{l:cover} that $\Omega=\mathbb{Z}^d\Sigma$ for some $\Sigma\subset
\Sigma_{s,t}$. For every $g=nak\in \Sigma$, we have $|n_{ij}|\le s$ and $a_i\le t a_{i-1}$.
It follows from (\ref{eq:compact}) that
$$
\|e_dg\|=\|e_da\|=a_d\ge \delta.
$$
Hence,
\begin{equation}\label{eq:ii}
a_i\ge t^{-1} a_{i+1}\ge \ldots\ge t^{-(d-i)} a_d\ge t^{i-d}\delta.
\end{equation}
Since $a_1\cdots a_d=1$, (\ref{eq:ii}) implies that all $a_i$'s are also bounded from above.
This proves that $\Sigma$ is a bounded subset of $\hbox{SL}_d(\mathbb{R})$, so that
$\Omega=\mathbb{Z}^d\Sigma$ is relatively compact. 

The converse statement is obvious.
\end{proof}

\vspace{0.5cm}

More generally, we consider $G(\mathbb{Z})\subset G(\mathbb{R})$ where $G$ is an algebraic
group defined over $\mathbb{Q}$. In many cases, $G(\mathbb{Z})$ is a lattice in
$G(\mathbb{R})$. In fact the following general criterion holds (see \cite{borel_a,pr}):

\begin{theorem}\label{th:gen}
Let $G$ be a connected algebraic group defined over $\mathbb{Q}$.
Then $G(\mathbb{Z})$ is a lattice in $G(\mathbb{R})$ if and only 
there are no nontrivial polynomial homomorphisms $G\to \mathbb{C}^\times$ defined over $\mathbb{Q}$.
\end{theorem}

Theorem \ref{th:gen} can be proven by generalising the construction of Siegel sets given above,
but for this one needs to develop more of structure theory of algebraic groups,
and we are not going to give a proof of this theorem here. Instead we prove a related 
result which also sheds some light into the structure of the space $\hbox{SL}_d(\mathbb{Z})\backslash \hbox{SL}_d(\mathbb{R})$. 

\begin{theorem}\label{th:orbit}
Let $G\subset \hbox{\rm SL}_d(\mathbb{C})$ be an algebraic group defined over $\mathbb{Q}$
which doesn't have any nontrivial polynomial homomorphisms $G\to \mathbb{C}^\times$ defined over
$\mathbb{Q}$.
Then the image of the map 
$$
\iota: G(\mathbb{Z})\backslash G(\mathbb{R})\to \mathcal{L}_d: g\mapsto \mathbb{Z}^d g
$$
is closed, and the map $\iota$ defines a homeomorphism
$G(\mathbb{Z})\backslash G(\mathbb{R})\simeq \hbox{\rm Im}(\iota)$.
\end{theorem}

For the proof, we need two lemmas.

\begin{lemma}\label{l:1}
For $G$ as in Theorem \ref{th:orbit}, there exist a polynomial homomorphism $\rho: \hbox{\rm
  GL}_d(\mathbb{C})\to \hbox{\rm GL}_N(\mathbb{C})$ defined over $\mathbb{Q}$ and $v\in\mathbb{Q}^N$
such that 
$$
G=\hbox{\rm Stab}_{\hbox{\rm\tiny GL}_d(\mathbb{C})}(v).
$$
\end{lemma}

\begin{proof}
Let
\begin{align*}
V_m&=\{P\in \mathbb{C}[x_{11},\ldots,x_{dd}]:\, \deg(P)\le m\},\\
W_m&=V_m\cap \mathcal{I}(G)=\{P\in V_m:\, P|_G=0\}.
\end{align*}
By the Hilbert Basis Theorem \cite[Th.~7.5]{am}, the ideal $\mathcal{I}(G)$ is finitely generated.
Hence, for sufficiently large $m$, it is generated by $W_m$, and we fix such $m$.
We consider the representation
$$
\sigma:\hbox{\rm GL}_d(\mathbb{C})\to \hbox{\rm GL}(V_m):\sigma(g):P\mapsto P(X\cdot g).
$$
We claim that
\begin{equation}\label{eq:ccc}
g\in G\quad \Longleftrightarrow \quad \sigma(g)(W_m)\subset W_m.
\end{equation}
Indeed, if $g\in G$, then for every $P\in W_m$, 
$$
\sigma(g)(P)(G)=P(G\cdot g)=P(G)=\{0\},
$$ 
so that $\sigma(g)(W_m)\subset W_m$. 
Conversely, if $\sigma(g)(W_m)\subset W_m$, then for every 
$P\in W_m$, 
$$
P(x\cdot g)\in \mathcal{I}(G)\quad\hbox{and}\quad P(g)=P(I\cdot g)=0.
$$
Since $W_m$ generates $\mathcal{I}(G)$, it follows that $g\in G$, as required.

Now we consider the wedge-product representation
$$
\rho=\wedge^{\dim(W_m)} \sigma: \hbox{GL}_d(\mathbb{C})\to \hbox{GL}(\wedge^{\dim(W_m)} V_m)
$$
and take a nonzero rational $v\in \wedge^{\dim(W_m)} W_m$. It follows from the properties
of the wedge-products that
\begin{equation}\label{eq:ccc2}
\sigma(g)(W_m)\subset W_m \quad \Longleftrightarrow \quad \rho(g)(v)\in \mathbb{C}v.
\end{equation}
Combining (\ref{eq:ccc}) and (\ref{eq:ccc2}), we deduce that $G$ is precisely
the stabiliser of the line $\mathbb{C}v$. Then we obtain a polynomial homomorphism
$\chi:G\to \mathbb{C}^\times$ defined over $\mathbb{Q}$. According to our assumption on $G$,
$\chi$ must be trivial, and this implies the lemma.
\end{proof} 

\begin{lemma}\label{l:cong}
Let $G$ be an algebraic group defined over $\mathbb{Q}$ and $\rho:G\to\hbox{\rm GL}_N(\mathbb{C})$
a polynomial represenation defined over $\mathbb{Q}$. Then $\rho(G(\mathbb{Z}))$ preserves a lattice $L$
contained in $\mathbb{Q}^N$.
\end{lemma}

\begin{proof}
We introduce the family of \emph{congruence subgroups} of $G(\mathbb{Z})$:
$$
\Gamma(m)=\{\gamma\in G(\mathbb{Z}):\, \gamma=I\mod m\}.
$$
It is clear that $\Gamma(m)$ is a finite-index normal subgroup of $G(\mathbb{Z})$.
We may write 
\begin{equation}\label{eq:coc}
\rho(I+X)=I+P(X),
\end{equation}
where $P$ is a polynomial map with rational coefficients such that $P(0)=0$.
We take an integer $m$ which is divisible by all denominators of the coefficients of $P$.
Then it follows from (\ref{eq:coc}) that $\rho(\Gamma(m))\subset \hbox{M}_N(\mathbb{Z})$
and $\rho(\Gamma(m))$ preserves $\mathbb{Z}^N$. Hence, $\rho(G(\mathbb{Z}))$ preserves
$L=\left<\mathbb{Z}^N\rho(G(\mathbb{Z}))\right>$. Since $|G(\mathbb{Z}):\Gamma(m)|<\infty$,
it is clear that $L$ is a lattice.
\end{proof} 

\begin{proof}[Proof of Theorem \ref{th:orbit}]
Consider a sequence of lattices $L_n=\mathbb{Z}^dg_n$ with $g_n\in G(\mathbb{R})$.
Suppose that $L_n\to L=\mathbb{Z}^dg$ with $g\in \hbox{SL}_d(\mathbb{R})$.
This means that $\gamma_ng_n\to g$ in $\hbox{SL}_d(\mathbb{R})$ for some sequence
$\gamma_n\in\hbox{SL}_d(\mathbb{Z})$. We take the representation $\rho$ and the rational vector $v$
constructed in Lemma \ref{l:1}. Then according to Lemma~\ref{l:cong}, 
$\rho(\hbox{\rm SL}_d(\mathbb{Z}))$ stabilises a lattice contained in $\mathbb{Q}^N$.
Since a multiple of $v$ is contained in this lattice, we conclude that 
the orbit $v\cdot \rho(\hbox{SL}_d(\mathbb{Z}))$ is discrete in $\mathbb{R}^N$.
Since
$$
v\cdot \rho(\gamma_n g_n)^{-1}=v\cdot \rho(\gamma_n)^{-1}\to v\cdot \rho(g)^{-1},
$$
it follows from discreteness that $v\cdot \rho(\gamma_n)^{-1}= v\cdot \rho(g)^{-1}$
for sufficiently large $n$. In particular, $v\cdot \rho(\gamma_n)^{-1}=v\cdot \rho(\gamma_{n_0})^{-1}$
and $\gamma_n=\gamma_{n_0}\delta_n$ for some $\delta_n\in \hbox{SL}_d(\mathbb{Z})\cap \hbox{Stab}(v)=G(\mathbb{Z})$.
Then $\delta_ng_n\to g'=\gamma_{n_0}^{-1}g$. Clearly, $g'\in G(\mathbb{R})$ and
$\mathbb{Z}^dg'=\mathbb{Z}^dg$.
This shows that $\hbox{Im}(\iota)$ is closed, and $G(\mathbb{Z})g_n\to G(\mathbb{Z})g'$ in
$G(\mathbb{Z})\backslash G(\mathbb{R})$, so that $\iota$ is a homeomorphism.
\end{proof}

Theorem can be used to construct examples of compact homogeneous spaces $G(\mathbb{Z})\backslash
G(\mathbb{R})$:

\begin{corollary}\label{c:cccc}
Let $G\subset \hbox{\rm GL}_d(\mathbb{C})$ be an algebraic group as in Theorem \ref{th:orbit}.
Suppose that there exists a $G$-invariant homogeneous polynomial $P\in \mathbb{Q}[x_1,\ldots,x_d]$
such that
$$
P(v)=0\;\; \Longleftrightarrow \;\; v=0\quad \quad \hbox{for $v\in \mathbb{Z}^d$.}
$$
Then the space $G(\mathbb{Z})\backslash G(\mathbb{R})$ is compact.
\end{corollary}

\begin{proof}
According to Theorem \ref{th:orbit}, it is sufficient to show that $\hbox{Im}(\iota)$ is relatively
compact. For this we apply Theorem \ref{th:mahler}. Suppose that there exist $g_n\in G(\mathbb{R})$
and $v_n\in\mathbb{Z}^d\backslash \{0\}$ such that $v_ng_n\to 0.$
Then $P(v_ng_n)=P(v_n)\to 0$. Since the set $P(\mathbb{Z}^d)$ is discrete, it follows that $P(v_n)=0$
for sufficiently large $n$. Then $v_n=0$, which gives a contradiction. Hence,
$\hbox{Im}(\iota)$ is relatively compact, as required.
\end{proof}

We illustrate Corollary \ref{c:cccc} by two examples:

\begin{itemize}
\item Let 
$$
Q(x)=\sum_{i,j=1}^d a_{ij} x_ix_j 
$$
be a nondegenerate quadratic form with rational coefficients and 
$$
G=\hbox{SO}(Q)=\{g\in\hbox{SL}_d(\mathbb{C}):\, Q(x\cdot g)=Q(x)\}
$$
the corresponding orthogonal group. Suppose that the equation $Q(x)=0$ has no
nonzero integral solutions. For instance, one can take $Q(x)=x_1^2+x_2^2-3x_3^2$.
Then according to Corollary \ref{c:cccc}, the space $G(\mathbb{Z})\backslash G(\mathbb{R})$ is compact. 
We note that if the equation $Q(x)=0$ has nonzero real solutions, then
the group $G(\mathbb{R})$ is not compact.

\item Fix $a,b\in\mathbb{N}$ such that the equation $w^2-ax^2-by^2+abz^2=0$ has no nonzero integral
  solutions.
Consider the matrices
$$
i=\left(
\begin{tabular}{ll}
$\sqrt{a}$ & $0$\\
$0$ & $-\sqrt{a}$
\end{tabular}
\right),
\quad
j=\left(
\begin{tabular}{ll}
$0$ & $1$\\
$b$ & $0$
\end{tabular}
\right),
\quad
k=\left(
\begin{tabular}{ll}
$0$ & $\sqrt{a}$\\
$-b\sqrt{a}$ & $0$
\end{tabular}
\right),
$$
which satisfy the quaternion relations
$$
i^2=aI,\quad j^2=bI,\quad i\cdot j=-j\cdot i=k.
$$
We claim that
$$
\Gamma=(\mathbb{Z}I+\mathbb{Z}i+\mathbb{Z}j+\mathbb{Z}k)\cap \hbox{SL}_2(\mathbb{R})
$$
is a cocompact lattice in $\hbox{SL}_2(\mathbb{R})$.

To check this, we note that $\{I,i,j,k\}$ forms a basis of $\hbox{M}_2(\mathbb{C})$.
We define the integral structure on $\hbox{M}_2(\mathbb{C})$ with respect to this basis,
and embed the group $G=\hbox{SL}_2(\mathbb{C})$ in $\hbox{GL}_2(\hbox{M}_2(\mathbb{C}))$
using the represenation $\rho:G\to \hbox{GL}_2(\hbox{M}_2(\mathbb{C}))$ defined by
$\rho(g):X\mapsto X\cdot g$. Then $\Gamma=G(\mathbb{Z})$ and $G(\mathbb{R})\simeq \hbox{SL}_2(\mathbb{R})$.
The polynomial 
$$
\det(wI+xi+yj+zk)=w^2-ax^2-by^2+ab z^2.
$$
is $G$-invariant, so that the claim follows from Corollary \ref{c:cccc}.

\end{itemize}

\section{Borel density theorem}

We conclude these lectures with a version of the Borel Density Theorem \cite{b}, which illustrates how
dynamical systems techniques can be used to address arithmetic questions.

\begin{theorem}[Borel density]\label{th:borel}
Let $\Gamma$ a lattice in $\hbox{\rm SL}_d(\mathbb{R})$. Then given a polynomial representation
$\rho:\hbox{\rm SL}_d(\mathbb{R})\to \hbox{\rm GL}_N(\mathbb{C})$,
every vector $v\in\mathbb{C}^N$ which is fixed by $\rho(\Gamma)$ is also fixed by $\rho(\hbox{\rm SL}_d(\mathbb{R}))$.
\end{theorem}

This theorem can be refined to show that the Zariski closure of 
$\Gamma$ is equal to $\hbox{SL}_d(\mathbb{C})$.
As we shall see, the proof that we present applies more generally if $\hbox{SL}_d(\mathbb{R})$ is replaced
by any Lie group $G\subset\hbox{\rm GL}_d(\mathbb{R})$ 
which is generated by unipotent one-parameter subgroups.

The main idea of the proof is to compare the recurrence property of orbits 
for measure-preserving actions (Lemma \ref{l:poincare})
with the rigid behaviour of orbits for polynomial actions (Lemma \ref{l:proj}).

\begin{lemma}[Poincare recurrence]\label{l:poincare}
Let $T:X\to X$ be a homeomorphism of a compact metric space $X$ and $\mu$ a Borel probability 
$T$-invariant measure on $X$. Then for $\mu$-almost every $x\in X$, $T^{n_k}(x)\to x$ along a subsequence
$n_k$.
\end{lemma}

Lemma \ref{l:poincare} is a standard fact from ergodic theory (see, for instance, \cite[Sec.~4.2]{bs}).

\begin{lemma}\label{l:proj}
Let $T\in\hbox{\rm GL}_d(\mathbb{C})$ be a unipotent element acting on a the projective space
$\mathbb{P}^{d-1}$. Suppose that for $[v]\in \mathbb{P}^{d-1}$, we have $T^{n_k}([v])\to [v]$
along a subsequence $n_k$. Then the vector $v$ is fixed by $T$.
\end{lemma}

Lemma \ref{l:proj} is a version of Corollary \ref{c:rec}, but it easy to prove it directly
using the Jordan Canonical Form for $T$.

\begin{proof}[Proof of Theorem \ref{th:borel}]
We consider the map
$$
\pi: \Gamma\backslash \hbox{SL}_d(\mathbb{R})\to \mathbb{P}^{N-1}:g\mapsto [v\cdot \rho(g)],
$$
and define a finite $\rho(\hbox{SL}_d(\mathbb{R}))$-invariant measure $\nu$ on $\mathbb{P}^{N-1}$
by 
$$
\nu(B)=\mu(\pi^{-1}(B))\quad\hbox{ for Borel $B\subset \mathbb{P}^{N-1}$,}
$$
where $\mu$ denotes  the finite invariant measure on $\Gamma\backslash \hbox{SL}_d(\mathbb{R})$.
We take a unipotent element $g\in \hbox{SL}_d(\mathbb{R})$ such that $g\ne I$.
The map $T=\rho(g)$ is also unipotent by Proposition \ref{th:us}, so that
by Lemma \ref{l:poincare}, $\nu$-almost every $x\in
\mathbb{P}^{N-1}$ is a limit point of the sequence $T^n(x)$. Hence, by Lemma \ref{l:proj}, 
for almost every $h\in G$, 
$$
v\cdot\rho(h)\rho(g)=v\cdot\rho(h).
$$
This implies that
the stabiliser of $v$ contains an infinite normal subgroup of $\hbox{SL}_d(\mathbb{R})$.
Hence, $v$ is fixed by $\rho(\hbox{SL}_d(\mathbb{R}))$, as required. 
\end{proof}

\section{Suggestions for further reading}

This exposition is intended to provide the reader with a first glimpse into
the beautiful theories of Lie groups, algebraic groups, and their discrete subgroups.
While we were trying to present some of the most important and ideas and techniques,
it is impossible to give a comprehensive treatment of these topics in a 10-hour
course. We hope that these notes would encourage the reader to study the subject in more details
and offer the following suggestions for further reading:

\begin{itemize}
\item the theory of Lie groups: \cite{dk,bump,hall,kir,on,ross,ser};
\item the theory of algebraic groups: \cite{borel_s} for a concise introduction
and \cite{borel,hum,sp} for a comprehensive treatment;
\item lattices in $\hbox{SL}_2(\mathbb{R})$: \cite{be,katok};
\item lattices in general Lie groups: \cite{morris,rag};
\item arithmetic lattices: \cite{borel_a,hum_a,ji,morris,pr}.
\end{itemize}

\noindent {\bf Acknowledgement:} I would like to thank the organisers of the Summer School 
``Modern dynamics and interactions with analysis, geometry and number theory''
for providing excellent working conditions.
The author was supported by the EPSRC grant EP/H000091/1 and the ERC grant 239606.

\vspace{0.5cm}

\noindent{\small \sc School of Mathematics, University of Bristol, UK\\ \texttt{a.gorodnik@bristol.ac.uk}}


\begin{thebibliography}{100}

\bibitem{am}
Atiyah, M. F.; Macdonald, I. G. \emph{Introduction to commutative algebra.} Addison-Wesley Publishing Co., Reading, Mass.-London-Don Mills, Ont. 1969.

\bibitem{be}
Benoist, Yves \emph{Pavages du plan.} Pavages, 1--48, Ed. \'Ec. Polytech., Palaiseau, 2001.

\bibitem{b}
Borel, Armand
\emph{Density properties for certain subgroups of semi-simple groups without compact components. }
Ann. of Math. 72 (1960), 179--188.

\bibitem{borel_s}
 Borel, Armand \emph{Linear algebraic groups.} 1966 Algebraic Groups and Discontinuous Subgroups (Proc. Sympos. Pure Math., Boulder, Colo., 1965) pp. 3--19 Amer. Math. Soc., Providence, R.I.


\bibitem{borel_a}
Borel, Armand \emph{Introduction aux groupes arithm\'etiques.}
Publications de l'Institut de Math\'ematique de l'Universit\'e de Strasbourg, XV. 
Actualit\'es Scientifiques et Industrielles, No. 1341 Hermann, Paris 1969.

\bibitem{borel}
 Borel, Armand \emph{Linear algebraic groups.} Second edition. Graduate Texts in Mathematics,
 126. Springer-Verlag, New York, 1991.

\bibitem{bs}
Brin, Michael; Stuck, Garrett
\emph{Introduction to dynamical systems.} Cambridge University Press, Cambridge, 2002.

\bibitem{dk}
 Duistermaat, J. J.; Kolk, J. A. C. \emph{Lie groups.} Universitext. Springer-Verlag, Berlin, 2000. 

\bibitem{bump}
Bump, Daniel \emph{Lie groups.} Graduate Texts in Mathematics, 225. Springer-Verlag, New York, 2004.

\bibitem{hall}
 Hall, Brian C. \emph{Lie groups, Lie algebras, and representations. An elementary introduction.} Graduate Texts in Mathematics, 222. Springer-Verlag, New York, 2003.

\bibitem{hum}
Humphreys, James E. 
\emph{Linear algebraic groups.}
 Graduate Texts in Mathematics, No. 21. Springer-Verlag, New York-Heidelberg, 1975.

\bibitem{hum_a}
Humphreys, James E. \emph{Arithmetic groups.} Lecture Notes in Mathematics, 789. Springer, Berlin, 1980. 

\bibitem{ji}
Ji, Lizhen
\emph{Arithmetic groups and their generalizations.}
What, why, and how. AMS/IP Studies in Advanced Mathematics, 43. American Mathematical Society, Providence, RI; International Press, Cambridge, MA, 2008.

\bibitem{katok}
Katok, Svetlana 
\emph{Fuchsian groups.}
 Chicago Lectures in Mathematics. University of Chicago Press, Chicago, IL, 1992.

\bibitem{kir}
Kirillov, Alexander, Jr. \emph{An introduction to Lie groups and Lie algebras.} Cambridge Studies in Advanced
Mathematics, 113. Cambridge University Press, Cambridge, 2008.

%\bibitem{lang}
%Lang, Serge \emph{Algebra.} Graduate Texts in Mathematics, 211. Springer-Verlag, New York, 2002.

\bibitem{morris} Morris, Dave \emph{Introduction to arithmetic groups.} Arxiv.math/0106063.

\bibitem{on}
Onishchik, A. L.; Vinberg, E. B. \emph{Lie groups and algebraic groups.} Springer Series in Soviet Mathematics. Springer-Verlag, Berlin, 1990. 

\bibitem{pr}
 Platonov, Vladimir; Rapinchuk, Andrei 
\emph{Algebraic groups and number theory.}
Pure and Applied Mathematics, 139. Academic Press, Inc., Boston, MA, 1994.

\bibitem{rag}
 Raghunathan, M. S. \emph{Discrete subgroups of Lie groups.}
 Ergebnisse der Mathematik und ihrer Grenzgebiete, Band 68. Springer-Verlag, New York-Heidelberg, 1972.

\bibitem{ross}
 Rossmann, Wulf \emph{Lie groups. An introduction through linear groups.} Oxford Graduate Texts in Mathematics, 5. Oxford University Press, Oxford, 2002.


\bibitem{ser}
Serre, Jean-Pierre \emph{Lie algebras and Lie groups.} Lecture Notes in Mathematics, 1500. Springer-Verlag, Berlin, 2006.

\bibitem{sp} Springer, T. A. \emph{Linear algebraic groups.}
 Second edition. Progress in Mathematics, 9. Birkh\"auser Boston, Inc., Boston, MA, 1998.

\end{thebibliography}
\end{document}